\documentclass[10pt,a4paper]{amsart}

\usepackage{amssymb}
\usepackage{amsmath}
\usepackage[all]{xy}
\usepackage{latexsym}
\usepackage{mathtools}
\usepackage{amsthm}
\usepackage{a4wide}
\usepackage{color}
\usepackage{hyperref}
\input diagxy

\theoremstyle{plain}
  \newtheorem{thm}{Theorem}[section]
  \newtheorem{lem}[thm]{Lemma}
  \newtheorem{prop}[thm]{Proposition}
  \newtheorem{cor}[thm]{Corollary}
\theoremstyle{definition}
  \newtheorem{defn}[thm]{Definition}
  
  \newtheorem{exmp}[thm]{Example}
  \newtheorem{rem}[thm]{Remark}

\makeatletter
\def\oarrowfill@#1#2#3#4#5{%
  $\m@th\thickmuskip0mu\medmuskip\thickmuskip\thinmuskip\thickmuskip
   \relax#5#1\mkern-7mu%
   \cleaders\hbox{$#5\mkern-2mu#2\mkern-2mu$}\hfill
   \mathclap{#3}\mathclap{#2}%
   \cleaders\hbox{$#5\mkern-2mu#2\mkern-2mu$}\hfill
   \mkern-7mu#4$%
}
\def\orightarrowfill@{%
  \oarrowfill@\relbar\relbar\circ\rightarrow}
\newcommand\xorightarrow[2][]{%
  \ext@arrow 0055{\orightarrowfill@}{#1}{#2}}
\makeatother

\newcommand{\da}{\downarrow}

\newcommand{\ua}{\uparrow}
\newcommand{\ra}{\rightarrow}
\newcommand{\la}{\leftarrow}
\newcommand{\Ra}{\rightarrow}
\newcommand{\La}{\leftarrow}
\newcommand{\lra}{\longrightarrow}
\newcommand{\lda}{\swarrow}
\newcommand{\rda}{\searrow}

\newcommand{\rat}{\!\rightarrowtail\!}

\newcommand{\olra}{\xorightarrow{\ \ \ \ }}

\newcommand{\bv}{\bigvee}
\newcommand{\bw}{\bigwedge}

\newcommand{\dv}{\dashv}
\newcommand{\rhu}{\rightharpoonup}
\newcommand{\ulc}{\ulcorner}
\newcommand{\urc}{\urcorner}

\newcommand{\ga}{\gamma}

\newcommand{\lam}{\lambda}

\newcommand{\CD}{\mathcal{D}}
\newcommand{\CF}{\mathcal{F}}
\newcommand{\CG}{\mathcal{G}}

\newcommand{\CK}{\mathcal{K}}

\newcommand{\CM}{\mathcal{M}}

\newcommand{\CP}{\mathcal{P}}
\newcommand{\CQ}{\mathcal{Q}}

\newcommand{\CT}{\mathcal{T}}
\newcommand{\CU}{\mathcal{U}}
\newcommand{\CV}{\mathcal{V}}
\newcommand{\BU}{{\bf U}}
\newcommand{\BY}{{\bf Y}}

\newcommand{\FD}{\frak{D}}

\newcommand{\bbA}{{\mathbb{A}}}
\newcommand{\bbB}{{\mathbb{B}}}
\newcommand{\bbC}{{\mathbb{C}}}

\newcommand{\bbX}{{\mathbb{X}}}
\newcommand{\sY}{{\sf Y}}
\newcommand{\skel}{{\sf skel}}

\newcommand{\colim}{{\rm colim}}

\newcommand{\dphi}{\phi^{\da}}
\newcommand{\uphi}{\phi_{\ua}}

\newcommand{\dpsi}{\psi^{\da}}
\newcommand{\upsi}{\psi_{\ua}}
\newcommand{\raphi}{\underline{\phi}}
\newcommand{\laphi}{\overline{\phi}}

\newcommand{\lapsi}{\overline{\psi}}
\newcommand{\CPd}{{\mathcal{P}}^{\dag}}

\newcommand{\PA}{\CP\bbA}
\newcommand{\PB}{\CP\bbB}
\newcommand{\PX}{\CP\bbX}
\newcommand{\PdA}{\CPd\bbA}
\newcommand{\PdB}{\CPd\bbB}
\newcommand{\sYd}{\sY^{\dag}}
\newcommand{\hF}{\ulc F \urc}
\newcommand{\hG}{\ulc G \urc}

\begin{document}

\title {Categories enriched over a quantaloid: Isbell adjunctions and Kan adjunctions}

\author{Lili Shen}
\address{Lili Shen\\School of Mathematics, Sichuan University, Chengdu 610064, China}
\email{scu@mickeylili.com}

\author{Dexue Zhang}
\address{Dexue Zhang\\School of Mathematics, Sichuan University, Chengdu 610064, China}
\email{dxzhang@scu.edu.cn}

\thanks{This work is supported by Natural Science Foundation of China (11071174).}

\begin{abstract}
Each distributor between categories enriched over a small quantaloid \(\mathcal{Q}\) gives rise to two adjunctions between the categories of contravariant and covariant presheaves, and hence to two monads. These two adjunctions are respectively generalizations of Isbell adjunctions and Kan extensions in category theory. It is proved that these two processes are functorial with infomorphisms playing as morphisms between distributors; and that the free cocompletion functor of \(\mathcal{Q}\)-categories factors through both of these functors.
\end{abstract}

\maketitle

\section{Introduction} \label{introduction}

A quantaloid \cite{Rosenthal1996,Stubbe_2005} is a category enriched over the symmetric monoidal closed category consisting of complete lattices and join-preserving functions. Since a quantaloid $\CQ$ is a closed and locally complete bicategory, one can develop a theory of categories enriched over $\CQ$ \cite{Benabou1967}. It should be stressed, that for such categories, coherence issues will not be a concern in most cases. For an overview of this theory the reader is referred to \cite{Heymans:2010:SQG:2049377,heymans2011symmetry,Stubbe_2005,Stubbe_2006}.

This paper is concerned with an extension of Isbell adjunctions and Kan extensions for $\CQ$-categories. In order to state the question clearly, we recall here Isbell adjunctions and Kan extensions in category theory.

Let $\bbA$ be a small category. The Isbell adjunction (or Isbell conjugacy) refers to the adjunction between ${\bf Set}^{\bbA^{\rm op}}$ and $({\bf Set}^{\bbA})^{\rm op}$ arising from the Yoneda embedding $\sY:\bbA\lra{\bf Set}^{\bbA^{\rm op}}$ and the co-Yoneda embedding $\sYd:\bbA\lra({\bf Set}^{\bbA})^{\rm op}$. Given  a functor $F:\bbA\lra\bbB$  between small categories, composition with $F$ induces a functor $-\circ F:{\bf Set}^{\bbB^{\rm op}}\lra{\bf Set}^{\bbA^{\rm op}}$. The functor $-\circ F$ has both a left and a right adjoint,  called respectively the left and the right Kan extension of $F$. Isbell adjunctions and Kan extensions have also been considered for categories enriched over a symmetric monoidal closed category  \cite{Borceux1994,Day2007651,kelly1982basic,kelly2005notes,Lawvere1973,lawvere1986taking}.

In this paper, it is shown that for a small quantaloid $\CQ$, each $\CQ$-distributor $\phi:\bbA\olra\bbB$ between $\CQ$-categories induces two adjunctions:
\begin{equation} \label{Isbell_intro}
\uphi\dv\dphi:\PA\rhu\PdB
\end{equation} and
\begin{equation} \label{Kan_intro}
\phi^*\dv\phi_*:\PB\rhu\PA,
\end{equation}
where $\PA$ and $\PdA$ are the counterparts of ${\bf Set}^{\bbA^{\rm op}}$ and $({\bf Set}^\bbA)^{\rm op}$, respectively.

If $\phi$ is the identity distributor on $\bbA$, then the adjunction $\uphi\dv\dphi$ reduces to the Isbell adjunction in \cite{Stubbe_2005}. Given a $\CQ$-functor $F:\bbA\lra\bbB$, consider the graph $F_\natural :\bbA\olra\bbB$ and the cograph $F^\natural :\bbB\olra\bbA$. Then it holds that (Theorem \ref{why_kan})
$$(F^{\natural})^*\dv (F^{\natural})_*= F^{\la}=(F_{\natural})^*\dv(F_{\natural})_*,$$
where $F^\la:\PB\lra\PA$ is the counterpart of the functor $-\circ F$ for $\CQ$-categories.

Therefore, the adjunctions (\ref{Isbell_intro}) and (\ref{Kan_intro}) extend the fundamental construction of  Isbell adjunctions and Kan extensions, so, they will be called Isbell adjunctions and Kan adjunctions by abuse of language. This paper is mainly concerned with the functoriality of these constructions.

For each $\CQ$-distributor $\phi:\bbA\olra\bbB$, the related Isbell adjunction and Kan adjunction give rise to a monad $\dphi\circ\uphi$ on $\PA$ (called a closure operator on $\bbA$ in this paper) and a monad $\phi_*\circ\phi^*$ on $\PB$, respectively. The correspondence
$$(\phi:\bbA\olra\bbB)\ \mapsto\ (\bbA,\dphi\circ\uphi)$$
is functorial from the category of $\CQ$-distributors and infomorphisms (defined below) to that of $\CQ$-closure spaces (a $\CQ$-category together with a closure operator) and continuous functors; and the correspondence
$$(\phi:\bbA\olra\bbB)\ \mapsto\ (\bbB,\phi_*\circ\phi^*)$$
defines a contravariant functor from the category of $\CQ$-distributors and infomorphisms  to that of $\CQ$-closure spaces. Furthermore,  the fixed points of the closure operator $\dphi\circ\uphi:\PA\lra\PA$ (or equivalently, all the algebras if we consider $\dphi\circ\uphi$ as a monad) constitute a complete $\CQ$-category $\CM(\phi)$; the fixed points of the closure operator $\phi_*\circ\phi^*:\PB\lra\PB$ also constitute a complete $\CQ$-category $\CK(\phi)$. Thus, each distributor $\phi:\bbA\olra\bbB$ generates two complete $\CQ$-categories: $\CM(\phi)$ and $\CK(\phi)$. It will be shown that $\CM$ is functorial and $\CK$ contravariant functorial  from the category of $\CQ$-distributors and infomorphisms to that of complete $\CQ$-categories and left adjoints. Moreover, the free cocompletion functor $\CP$ of $\CQ$-categories factors through both $\CM$ and $\CK$.

It should be pointed out that some conclusions in this paper have been proved, in the circumstance of concept lattices, in \cite{Shen2013166} for discrete $\CQ$-categories in the case that $\CQ$ is an one-object quantaloid, i.e., a unital quantale. The situation dealt with here is much more involved, and the method developed here allows for a wide range of applicability.

\section{Categories enriched over a quantaloid} \label{quantaloid}

The theory of categories enriched over a quantaloid has been studied systematically in \cite{Stubbe_2005,Stubbe_2006}. In this section, we recall some basic concepts and fix some notations that will be used in the sequel.

Complete lattices and join-preserving functions constitute a symmetric monoidal closed category {\bf Sup}. A {\it quantaloid} $\CQ$ is a {\bf Sup}-enriched category \cite{Rosenthal1996,Stubbe_2005}. Explicitly, a quantaloid $\CQ$ is a category with a class of objects $\CQ_0$ such that $\CQ(A,B)$ is a complete lattice for all $A,B\in\CQ_0$, and the composition $\circ$ of morphisms preserves joins in both variables, i.e.,
$$g\circ\Big(\bv_{i}f_i\Big)=\bv_{i}(g\circ f_i)\quad\text{and}\quad\Big(\bv_{j}g_j\Big)\circ f=\bv_{j}(g_j\circ f)$$
for all $f,f_i\in\CQ(A,B)$ and $g,g_j\in\CQ(B,C)$. The complete lattice $\CQ(A,B)$ has a top element $\top_{A,B}$ and a bottom element $\bot_{A,B}$.

In this paper, $\CQ$ is always assumed to be a small  quantaloid, i.e., $\CQ_0$ is a set.

For each $X\in\CQ_0$ and $f\in\CQ(A,B)$, both functions
$$-\circ f:\CQ(B,X)\lra\CQ(A,X):g\mapsto g\circ f,$$
$$f\circ -:\CQ(X,A)\lra\CQ(X,B):g\mapsto f\circ g$$
have respective right adjoints:
$$-\lda f:\CQ(A,X)\lra\CQ(B,X):g\mapsto g\lda f,$$
$$f\rda -:\CQ(X,B)\lra\CQ(X,A):g\mapsto f\rda g.$$
The operators $\lda$ and $\rda$ are respectively the {\it left} and {\it right implications}.

A {\it $\CQ$-category} \cite{Stubbe_2005} $\bbA$ consists of a set $\bbA_0$ equipped with a map $t:\bbA_0\lra\CQ_0:x\mapsto tx$ ($tx$ is called the {\it type} of $x$ and $\bbA_0$ is called a {\it $\CQ$-typed set}) and hom-arrows $\bbA(x,y)\in\CQ(tx,ty)$ such that
\begin{itemize}
\item[\rm (1)] $1_{tx}\leq\bbA(x,x)$ for all $x\in\bbA_0$;
\item[\rm (2)] $\bbA(y,z)\circ\bbA(x,y)\leq\bbA(x,z)$ for all $x,y,z\in\bbA_0$.
\end{itemize}

A {\it $\CQ$-functor} \cite{Stubbe_2005} $F:\bbA\lra\bbB$ between $\CQ$-categories is a map $F:\bbA_0\lra\bbB_0$ such that
\begin{itemize}
\item[\rm (1)] $F$ is {\it type-preserving} in the sense that $\forall x\in\bbA_0$, $tx=t(Fx)$;
\item[\rm (2)] $\forall x,x'\in\bbA_0$, $\bbA(x,x')\leq\bbB(Fx,Fx')$.
\end{itemize}

A $\CQ$-functor $F:\bbA\lra\bbB$ is {\it fully faithful} if $\bbA(x,x')=\bbB(Fx,Fx')$  for all $x,x'\in\bbA_0$. Bijective fully faithful $\CQ$-functors are exactly the isomorphisms in the category $\CQ$-{\bf Cat} of $\CQ$-categories and $\CQ$-functors.

A $\CQ$-category $\bbB$ is a (full) $\CQ$-subcategory of $\bbA$ if $\bbB_0$ is a subset of $\bbA_0$ and  $\bbB(x,y)=\bbA(x,y)$ for all $x,y\in\bbB_0$.

Given a $\CQ$-category $\bbA$, there is a natural underlying preorder $\leq$ on $\bbA_0$. For $x,y\in\bbA_0$, $x\leq y$ if and only if they are of the same type $tx=ty=A$ and $1_A\leq\bbA(x,y)$. Two objects $x,y$ in $\bbA$ are {\it isomorphic} if $x\leq y$ and $y\leq x$, written $x\cong y$. $\bbA$ is {\it skeletal} if no two different objects in $\bbA$ are isomorphic.

The underlying preorders on $\CQ$-categories induce an order between $\CQ$-functors:
$$F\leq G:\bbA\lra\bbB\iff\forall x\in\bbA_0,Fx\leq Gx\ \text{in}\ \bbB_0.$$
We denote $F\cong G:\bbA\lra\bbB$ if $F\leq G$ and $G\leq F$.

A pair of $\CQ$-functors $F:\bbA\lra\bbB$ and $G:\bbB\lra\bbA$ form an {\it adjunction} \cite{Stubbe_2005}, written $F\dv G:\bbA\rhu\bbB$, if $1_{\bbA}\leq G\circ F$ and $F\circ G\leq 1_{\bbB}$, where $1_\bbA$ and $1_\bbB$ are respectively the identity $\CQ$-functors on $\bbA$ and $\bbB$. In this case, $F$ is called a {\it left adjoint} of $G$ and $G$ a {\it right adjoint} of $F$.

A {\it $\CQ$-distributor} \cite{Stubbe_2005} $\phi:\bbA\olra\bbB$ between $\CQ$-categories is a map that assigns to each pair  $(x,y)\in\bbA_0\times\bbB_0$ a morphism $\phi(x,y)\in\CQ(tx,ty)$ in $\CQ$, such that
\begin{itemize}
\item[\rm (1)] $\forall x\in\bbA_0$, $\forall y,y'\in\bbB_0$, $\bbB(y',y)\circ\phi(x,y')\leq\phi(x,y)$;
\item[\rm (2)] $\forall x,x'\in\bbA_0$, $\forall y\in\bbB_0$, $\phi(x',y)\circ\bbA(x,x')\leq\phi(x,y)$.
\end{itemize}

$\CQ$-categories and $\CQ$-distributors constitute a quantaloid $\CQ$-{\bf Dist} \cite{Stubbe_2005} in which
\begin{itemize}
\item the local order is defined pointwise, i.e., for $\CQ$-distributors $\phi,\psi:\bbA\olra\bbB$,
    $$\phi\leq\psi\iff\forall x\in\bbA_0,\forall y\in\bbB_0,\phi(x,y)\leq\psi(x,y);$$
\item the composition $\psi\circ\phi:\bbA\olra\bbC$ of $\CQ$-distributors $\phi:\bbA\olra\bbB$ and $\psi:\bbB\olra\bbC$ is given by
    $$\forall x\in\bbA_0,\forall z\in\bbC_0,(\psi\circ\phi)(x,z)=\bv_{y\in\bbB_0}\psi(y,z)\circ\phi(x,y);$$
\item the identity $\CQ$-distributor on a $\CQ$-category $\bbA$ is the hom-arrows of $\bbA$ and will be denoted by $\bbA:\bbA\olra\bbA$;
\item for $\CQ$-distributors $\phi:\bbA\olra\bbB$, $\psi:\bbB\olra\bbC$ and $\eta:\bbA\olra\bbC$, the left implication $\eta\lda\phi:\bbB\olra\bbC$ and the right implication $\psi\rda\eta:\bbA\olra\bbB$ are given by
    $$\forall y\in\bbB_0,\forall z\in\bbC_0,(\eta\lda\phi)(y,z)=\bw_{x\in\bbA_0}\eta(x,z)\lda\phi(x,y)$$
    and
    $$\forall x\in\bbA_0,\forall y\in\bbB_0,(\psi\rda\eta)(x,y)=\bw_{z\in\bbC_0}\psi(y,z)\rda\eta(x,z).$$
\end{itemize}

An {\it adjunction} \cite{Stubbe_2005} in a quantaloid $\CQ$,  $f\dv g:A\rhu B$ in symbols,  is a pair of morphisms $f:A\lra B$ and $g:B\lra A$ in $\CQ$ such that $1_A\leq g\circ f$ and $f\circ g\leq 1_B$. In this case, $f$ is a left adjoint of $g$ and $g$ a right adjoint of $f$.  In particular, a pair of $\CQ$-distributors $\phi:\bbA\olra\bbB$ and $\psi:\bbB\olra\bbA$ form an adjunction $\phi\dv\psi:\bbA\rhu\bbB$ in the quantaloid $\CQ$-{\bf Dist} if $\bbA\leq\psi\circ\phi$ and $\phi\circ\psi\leq\bbB$.

Every $\CQ$-functor $F:\bbA\lra\bbB$ induces an adjunction $F_{\natural}\dv F^{\natural}:\bbA\rhu\bbB$ in $\CQ$-{\bf Dist} with $F_{\natural}(x,y)=\bbB(Fx,y)$ and $F^{\natural}(y,x)=\bbB(y,Fx)$ for all $x\in\bbA_0$ and $y\in\bbB_0$. The $\CQ$-distributors $F_{\natural}:\bbA\olra\bbB$ and $F^{\natural}:\bbB\olra\bbA$ are called the {\it graph} and {\it cograph} of $F$, respectively.
\begin{prop} {\rm\cite{Heymans:2010:SQG:2049377}} \label{graph_cograph_implication}
If $f\dv g:A\rhu B$ in a quantaloid $\CQ$, then the following identities hold for all $\CQ$-arrows $h,h'$ whenever the compositions and implications make sense:
\begin{itemize}
\item[\rm (1)] $h\circ f=h\lda g$, $g\circ h=f\rda h$.
\item[\rm (2)] $(f\circ h)\rda h'=h\rda(g\circ h')$, $(h'\circ f)\lda h=h'\lda(h\circ g)$.
\item[\rm (3)] $(h\rda h')\circ f=h\rda(h'\circ f)$, $g\circ(h'\lda h)=(g\circ h')\lda h$.
\item[\rm (4)] $g\circ(h\rda h')=(h\circ f)\rda h'$, $(h'\lda h)\circ f=h'\lda(g\circ h)$.
\end{itemize}
\end{prop}

The identities in Proposition \ref{graph_cograph_implication} will be frequently applied in the next sections to the adjunction $F_{\natural}\dv F^{\natural}:\bbA\rhu\bbB$ induced by a $\CQ$-functor $F:\bbA\lra\bbB$.

\begin{prop} {\rm\cite{Stubbe_2005}}
Let $F:\bbA\lra\bbB$ and $G:\bbB\lra\bbA$ be a pair of $\CQ$-functors. The following conditions are equivalent:
\begin{itemize}
\item[\rm (1)] $F\dv G:\bbA\rhu\bbB$.
\item[\rm (2)] $F_{\natural}=\bbB(F-,-)=\bbA(-,G-)=G^{\natural}$.
\item[\rm (3)] $G_{\natural}\dv F_{\natural}:\bbB\rhu\bbA$ in $\CQ$-{\bf Dist}.
\item[\rm (4)] $G^{\natural}\dv F^{\natural}:\bbA\rhu\bbB$ in $\CQ$-{\bf Dist}.
\end{itemize}
\end{prop}

\begin{prop} \label{fully_faithful_graph_cograph}
Let $F:\bbA\lra\bbB$ be a $\CQ$-functor.
\begin{itemize}
\item[\rm (1)] If $F$ is fully faithful, then $F^{\natural}\circ F_{\natural}=\bbA$.
\item[\rm (2)] If $F$ is essentially surjective in the sense that there is some $x\in\bbA_0$ such that $Fx\cong y$ in $\bbB$ for all $y\in\bbB_0$, then $F_{\natural}\circ F^{\natural}=\bbB$.
\end{itemize}
\end{prop}

\begin{proof}
(1) If $F$ is fully faithful, then for all $x,x'\in\bbA_0$,
$$(F^{\natural}\circ F_{\natural})(x,x')=\bv_{y\in\bbB_0}\bbB(y,Fx')\circ\bbB(Fx,y)=\bbB(Fx,Fx')=\bbA(x,x').$$

(2) If $F$ is essentially surjective, then for all $y,y'\in\bbB_0$, there is some $x\in\bbA_0$ such that $Fx\cong y$. Thus
\begin{align*}
(F_{\natural}\circ F^{\natural})(y,y')&=\bv_{a\in\bbA_0}\bbB(Fa,y')\circ\bbB(y,Fa)\\
&\geq\bbB(Fx,y')\circ\bbB(y,Fx)\\
&=\bbB(y,y')\circ\bbB(y,y)\\
&\geq\bbB(y,y').
\end{align*}
Since $F_{\natural}\circ F^{\natural}\leq\bbB$ holds trivially, it follows that $F_{\natural}\circ F^{\natural}=\bbB$.
\end{proof}

Following \cite{Stubbe_2005}, for each $X\in\CQ_0$, write $*_X$ for the $\CQ$-category with only one object $*$ of type $t*=X$ and hom-arrow $1_X$.

A {\it contravariant presheaf} \cite{Stubbe_2005} on  a $\CQ$-category $\bbA$ is a $\CQ$-distributor $\mu:\bbA\olra *_X$  with $X\in\CQ_0$. Contravariant presheaves on a $\CQ$-category $\bbA$ constitute a $\CQ$-category $\PA$ in which
$$t\mu=X\quad\text{and}\quad\PA(\mu,\lam)=\lam\lda\mu$$
for all $\mu:\bbA\olra *_X$ and $\lam:\bbA\olra *_{Y}$ in $(\PA)_0$.

Dually,  a {\it covariant presheaf} on a $\CQ$-category $\bbA$ is a $\CQ$-distributor  $\mu:*_X\olra\bbA$. Covariant presheaves on $\bbA$ constitute a $\CQ$-category $\PdA$ in which
$$t\mu=X\quad\text{and}\quad\PdA(\mu,\lam)=\lam\rda\mu$$
for all $\mu:*_X\olra\bbA$ and $\lam:*_{Y}\olra\bbA$.

In particular, we denote $\CP(*_X)=\CP X$ and $\CPd (*_X)=\CPd X$ for each $X\in\CQ_0$.

\begin{rem} \label{PdA_QDist_order}
For each $\CQ$-category $\bbA$, it follows from the definition that the underlying preorder in $\PA$ coincides with the local order in $\CQ$-{\bf Dist}, while the underlying preorder in $\PdA$ is  the {\it reverse} local order in $\CQ$-{\bf Dist}. That is to say, for all $\mu,\lam\in\PdA$, we have
$$\mu\leq\lam\ \text{in}\ (\PdA)_0\iff\lam\leq\mu\ \text{in}\ \CQ\text{-}{\bf Dist}.$$
In order to get rid of the confusion about the symbol $\leq$, from now on we make the convention that the symbol $\leq$ between $\CQ$-distributors always denotes the local order in $\CQ$-{\bf Dist} if not otherwise specified.
\end{rem}

Given a $\CQ$-category $\bbA$ and $a\in\bbA_0$, write $\sY a$ for the $\CQ$-distributor
$$\bbA\olra *_{ta},\quad x\mapsto\bbA(x,a);$$
write $\sYd a$ for the $\CQ$-distributor
$$*_{ta}\olra\bbA,\quad x\mapsto\bbA(a,x).$$

The following  lemma implies that both $\sY:\bbA\lra\PA, a\mapsto\sY a$ and $\sYd:\bbA\lra\PdA, a\mapsto\sYd a$ are fully faithful $\CQ$-functors (hence embeddings if $\bbA$ is skeletal). Thus, $\sY$ and $\sYd$ are called respectively the {\it Yoneda embedding} and the {\it co-Yoneda embedding}.

\begin{lem}[Yoneda] \label{Yoneda_lemma} {\rm\cite{Stubbe_2005}}
$\PA(\sY a,\mu)=\mu(a)$ and $\PdA(\lam,\sYd a)=\lam(a)$
for all $a\in\bbA_0$, $\mu\in\PA$ and $\lam\in\PdA$.
\end{lem}

For each $\CQ$-distributor $\phi:\bbA\olra\bbB$ and $x\in\bbA_0,y\in\bbB_0$, write $\phi(x,-)$ for the $\CQ$-distributor  $\phi\circ\sYd_{\bbA}x: *_{tx}\olra\bbA\olra\bbB$; and write $\phi(-,y)$ for the $\CQ$-distributor ${\sf Y}_{\bbB}y\circ\phi:\bbA\olra\bbB\olra*_{ty}$.
Then the Yoneda lemma can be phrased as the commutativity of the following diagrams:
$$\bfig
\ptriangle/->`<-`<-/[\PA`*_{t\mu}`\bbA;\PA(-,\mu)`\sY_{\natural}`\mu] \place(250,500)[\circ]\place(0,250)[\circ]\place(250,250)[\circ]
\ptriangle(1200,0)/<-`->`->/[\PdA`*_{t\lam}`\bbA;\PdA(\lam,-)`(\sYd)^{\natural}`\lam] \place(1450,500)[\circ]\place(1200,250)[\circ]\place(1450,250)[\circ]
\efig$$
That is, $\mu=\PA(\sY-,\mu)$ and $\lam=\PdA(\lam,\sYd-)$.

\begin{rem} \label{distributor_notion}
Given $\CQ$-distributors $\phi:\bbA\olra\bbB$, $\psi:\bbB\olra\bbC$ and $\eta:\bbA\olra\bbC$, one can form $\CQ$-distributors such as $\phi(x,-),\eta\lda\phi,\eta\rda\psi(y,-)$,  etc.
We list here some basic  formulas related to these $\CQ$-distributors that will be used in the sequel.
$$\forall x\in\bbA_0,\forall z\in\bbC_0,(\psi\circ\phi)(x,z)=\psi(-,z)\circ\phi(x,-);$$
$$\forall x\in\bbA_0,(\psi\circ\phi)(x,-)=\psi\circ\phi(x,-);$$
$$\forall y\in\bbB_0, (\eta\lda\phi)(y,-)=\eta\lda\phi(-,y);$$
$$\forall x\in\bbA_0,\forall y\in\bbB_0,(\psi\rda\eta)(x,y)=\psi(y,-)\rda\eta(x,-).$$
\end{rem}

For a $\CQ$-functor $F:\bbA\lra\bbB$ between $\CQ$-categories, define $\CQ$-functors $F^{\ra}:\PA\lra\PB$ and $F^{\la}:\PB\lra\PA$ by $F^{\ra}(\mu)=\mu\circ F^{\natural}$ and $F^{\la}(\lam)=\lam\circ F_{\natural}$. Then
$$F^{\ra}\dv F^{\la}:\PA\rhu\PB$$
in $\CQ$-{\bf Cat}. For all $\lam\in\PB$ and $x\in\bbA_0$, it can be verified that
\begin{equation} \label{F_la_lam_Fx}
F^{\la}(\lam)(x)=\lam(Fx)\in \CQ(tx,t\lam).
\end{equation}

Dually, we may also define $\CQ$-functors $F^{\Ra}:\PdA\lra\PdB$ and $F^{\La}:\PdB\lra\PdA$ by $F^{\Ra}(\mu)=F_{\natural}\circ\mu$ and $F^{\La}(\lam)=F^{\natural}\circ\lam$. Then $$F^{\La}\dv F^{\Ra}:\PdB\rhu\PdA$$
in $\CQ$-{\bf Cat}. For all $\lam\in\PdB$ and $x\in\bbA_0$, it can be verified that
\begin{equation} \label{F_La_lam_Fx}
F^{\La}(\lam)(x)=\lam(Fx)\in\CQ(t\lam, tx).
\end{equation}

Note that the symbol $F^\ra$ is used for both of the $\CQ$-functors $\PA\lra\PB$ and $\PdA\lra\PdB$. This should cause no confusion since it can be easily detected from the context which one it stands for. So is the symbol $F^\la$.

We would like to stress that
\begin{equation} \label{Fla_Fra_adjuntion}
\mu\leq F^{\la}\circ F^{\ra}(\mu)\quad\text{and}\quad F^{\ra}\circ F^{\la}(\lam)\leq\lam
\end{equation}
for all $\mu\in\PA$ and $\lam\in\PB$; whereas
\begin{equation} \label{FLa_FRa_adjuntion}
\nu\leq F^{\La}\circ F^{\Ra}(\nu)\quad\text{and}\quad F^{\Ra}\circ F^{\La}(\ga)\leq\ga
\end{equation}
for all $\nu\in\PdA$ and $\gamma\in\PdB$ by Remark \ref{PdA_QDist_order}.

For a $\CQ$-functor $F:\bbA\lra\bbB$ and a contravariant presheaf $\mu\in\PA$, the {\it colimit of $F$ weighted by $\mu$} \cite{Stubbe_2005} is an object $\colim_{\mu}F\in\bbB_0$  (necessarily of type $t\mu$) such that
$$\bbB({\colim}_{\mu}F,-)=F_{\natural}\lda\mu.$$

Dually, for a covariant presheaf $\lam\in\PdA$, the {\it limit of $F$ weighted by $\lam$} is an object $\lim_{\lam}F\in\bbB_0$ (necessarily of type $t\lam$) such that
$$ \bbB(-,{\lim}_{\lam}F)=\lam\rda F^{\natural}.$$

A $\CQ$-category $\bbB$ is {\it cocomplete} (resp. {\it complete}) if $\colim_{\mu}F$ (resp. $\lim_{\lam}F$) exists for each $\CQ$-functor $F:\bbA\lra\bbB$ and $\mu\in\PA$ (resp. $\lam\in\PdA$).

In particular, for a $\CQ$-category $\bbA$  and $\mu\in\PA$ (resp. $\lam\in\PdA$), the colimit $\colim_{\mu}1_{\bbA}$ (the limit ${\lim}_{\lam}1_{\bbA}$, resp.) exists if there is some $a\in\bbA_0$ such that
$$\bbA(a,-)=\bbA\lda\mu\quad(\text{resp.}\ \bbA(-,a)=\lam\rda\bbA).$$
In this case, we say that $a$ is a {\it supremum} of $\mu$ (resp. an {\it infimum} of $\lam$), and denote it by $\sup\mu$ (resp. $\inf\lam$). Note that for any $\CQ$-functor $F:\bbA\lra\bbB$ and $\mu\in\PA$  (resp. $\lam\in\PdA$),
$$\colim_{\mu}F={\sup}_{\bbB}F^{\ra}(\mu)\quad(\text{resp.}\ {\lim}_{\lam}F={\inf}_{\bbB}F^{\ra}(\lam))$$
when it exists.

Let $\bbA$ be a $\CQ$-category. For $x\in\bbA_0$ and $f\in\CP(tx)$ (resp. $f\in\CPd(tx)$), the {\it tensor} (resp. {\it cotensor}) \cite{Stubbe_2006} of $f$ and $x$, denoted by $f\otimes x$ (resp. $f\rat x$), is an object in $\bbA_0$ of type $t(f\otimes x)=tf$ (resp. $t(f\rat x)=tf$) such that
$$\bbA(f\otimes x,-)=\bbA(x,-)\lda f\quad(\text{resp}.\ \bbA(-,f\rat x)=f\rda\bbA(-,x)).$$

For $x\in\bbA_0$ and $f\in\CP(tx)$, it is easily seen that the tensor $f\otimes x$ is exactly the supremum of $f\circ\sY x\in\PA$ if it exists. Dually, for $y\in\bbA_0$ and $g\in\CPd(ty)$, the cotensor $g\rat y$ is the infimum of $\sYd y\circ g\in\PdA$ if it exists.

A $\CQ$-category $\bbA$ is said to be {\it tensored} (resp. {\it cotensored}) if the tensor $f\otimes x$ (resp. the cotensor $f\rat x$) exists for all choices of $x$ and $f$.

\begin{exmp} \label{PA_tensor}
Let $\bbA$ be a $\CQ$-category.
\begin{itemize}
\item[\rm (1)] $\PA$ is a tensored and cotensored $\CQ$-category in which
$$f\otimes\mu=f\circ\mu,\quad g\rat\mu=g\rda\mu$$
for all $\mu\in\PA$ and $f\in\CP(t\mu)$, $g\in\CPd(t\mu)$.
\item[\rm (2)] $\PdA$ is a tensored and cotensored $\CQ$-category in which
$$f\otimes\lam=\lam\lda f,\quad g\rat\lam=\lam\circ g$$
for all $\lam\in\PdA$ and $f\in\CP(t\lam)$, $g\in\CPd(t\lam)$.
\end{itemize}
\end{exmp}

Let $\bbA$ be a $\CQ$-category and $X\in\CQ_0$. The objects in $\bbA$ with type $X$ constitute a subset of the underlying preordered set $\bbA_0$ and we denote it by $\bbA_X$. A $\CQ$-category $\bbA$ is said to be {\it order-complete} \cite{Stubbe_2006} if each $\bbA_X$ admits all joins  in the underlying preorder.

For each subset $\{x_i\}\subseteq\bbA_X$, if the join (resp. meet) of $\{x_i\}$ in $\bbA_X$ exists, then
$$\bv_i x_i=\sup\bv_i\sY x_i\quad\Big(\text{resp.}\quad\bw_i x_i=\inf\bw_i\sYd x_i\Big),$$
where  $\displaystyle\bv_i x_i$ and $\displaystyle\bw_i x_i$ denote respectively the join and the meet  in $\bbA_X$;  $\displaystyle\bv_i\sY x_i$ denotes the join in $(\PA)_X$ and $\displaystyle\bw_i\sYd x_i$ the meet in $(\PdA)_X$.

\begin{thm} {\rm\cite{Stubbe_2005,Stubbe_2006}} \label{complete_cocomplete_equivalent}
For a $\CQ$-category $\bbA$, the following conditions are equivalent:
\begin{itemize}
\item[\rm (1)] $\bbA$ is complete.
\item[\rm (2)] $\bbA$ is cocomplete.
\item[\rm (3)] $\bbA$ is tensored, cotensored, and order-complete.
\item[\rm (4)] Each $\mu\in\PA$ has a supremum.
\item[\rm (5)] $\sY$ has a left adjoint in $\CQ$-{\bf Cat}, given by $\sup:\PA\lra\bbA$.
\item[\rm (6)] Each $\lam\in\PdA$ has an infimum.
\item[\rm (7)] $\sYd$ has a right adjoint in $\CQ$-{\bf Cat}, given by $\inf:\PdA\lra\bbA$.
\end{itemize}
In this case, for each $\mu\in\PA$ and $\lam\in\PdA$,
$$\sup\mu=\bv_{a\in\bbA_0}(\mu(a)\otimes a),\quad\inf\lam=\bw_{a\in\bbA_0}(\lam(a)\rat a),$$
where $\displaystyle\bv$ and $\displaystyle\bw$ denote respectively the join in $\bbA_{t\mu}$ and the meet in $\bbA_{t\lam}$.
\end{thm}

\begin{exmp} \label{PX_PA_complete}
Let $\bbA$ be a $\CQ$-category.
\begin{itemize}
\item[\rm (1)] $\PA$ is a complete $\CQ$-category in which
$$\sup\Phi=\bv_{\mu\in\PA}\Phi(\mu)\circ\mu=\Phi\circ(\sY_{\bbA})_{\natural}$$
$$\bfig
\ptriangle/->`<-`<-/[\PA`*_{t\Phi}`\bbA; \Phi`(\sY_{\bbA})_{\natural}`\sup\Phi] \place(250,500)[\circ] \place(0,250)[\circ]\place(250,250)[\circ]
\efig$$
for all $\Phi\in\CP(\PA)$ \cite{Stubbe_2005} and
$$\inf\Psi=\bw_{\mu\in\PA}\Psi(\mu)\rda\mu=\Psi\rda(\sY_{\bbA})_{\natural}$$
$$\bfig
\btriangle[\bbA`*_{t\Psi}`\PA; \inf\Psi`(\sY_{\bbA})_{\natural}`\Psi]
\place(250,0)[\circ]\place(0,250)[\circ] \place(250,250)[\circ] \place(125,125)[\twoar(1,1)]
\efig$$
for all $\Psi\in\CPd(\PA)$, i.e., $\inf\Psi$ is the largest $\CQ$-distributor $\mu:\bbA\olra *_{t\Psi}$ such that $\Psi\circ\mu\leq (\sY_{\bbA})_{\natural}$.
\item[\rm (2)] $\PdA$ is a complete $\CQ$-category in which
$$\sup\Phi=(\sYd_{\bbA})^{\natural}\lda\Phi\quad\text{and} \quad\inf\Psi=(\sYd_{\bbA})^{\natural}\circ\Psi$$
for all $\Phi\in\CP(\PdA)$ and $\Psi\in\CPd(\PdA)$.
\end{itemize}

In particular, $\CP X$ and $\CPd X$ are both complete $\CQ$-categories for all $X\in\CQ_0$.
\end{exmp}

\begin{prop} \label{F_la_ra_condition} {\rm\cite{Stubbe_2006}}
Let $F:\bbA\lra\bbB$ be a $\CQ$-functor between $\CQ$-categories, with $\bbA$ complete, then $F$ is a left (resp. right) adjoint in $\CQ$-{\bf Cat} if and only if
\begin{enumerate}
\item[{\rm (1)}] $F$ preserves tensors (resp. cotensors) in the sense that $F(f\otimes_{\bbA}x)= f\otimes_{\bbB}Fx$ (resp. $F(f\rat_{\bbA}x)=f\rat_{\bbB}Fx$) for all $x\in\bbA_0$ and $f\in\CP(tx)$ (resp. $f\in\CPd(tx)$).
\item[{\rm (2)}] For all $X\in\CQ_0$, $F:\bbA_X\lra\bbB_X$ preserves arbitrary joins (resp. meets).
\end{enumerate}
\end{prop}

\begin{cor} {\rm\cite{Stubbe_2006}} \label{left_adjoint_preserves_sup}
Let $F:\bbA\lra\bbB$ be a $\CQ$-functor between $\CQ$-categories, with $\bbA$ complete, then $F:\bbA\lra\bbB$ is a left (resp. right) adjoint if and only if $F$ preserves supremum (resp. infimum) in the sense that $F(\sup_{\bbA}\mu)=\sup_{\bbB}F^{\ra}(\mu)$ for all $\mu\in\PA$ (resp. $F(\inf_{\bbA}\mu)=\inf_{\bbB}F^{\Ra}(\mu)$ for all $\mu\in\PdA$).
\end{cor}

Thus, left (resp. right) adjoint $\CQ$-functors between complete $\CQ$-categories are exactly suprema-preserving (resp. infima-preserving) $\CQ$-functors. Complete $\CQ$-categories and left adjoint $\CQ$-functors constitute a subcategory of $\CQ$-{\bf Cat} which will be  denoted by $\CQ$-{\bf CCat}.

The forgetful functor $\CQ\text{-}{\bf CCat}\lra\CQ\text{-}{\bf Cat}$ has a left adjoint $\CP:\CQ\text{-}{\bf Cat}\lra\CQ\text{-}{\bf CCat}$ that sends a $\CQ$-functor $F:\bbA\lra\bbB$ to the left adjoint $\CQ$-functor $F^{\ra}:\PA\lra\PB$. This implies that $\CP\bbA$ is the free cocompletion  of $\bbA$ \cite{Stubbe_2005}.

Now, we introduce the crucial notion in this paper, that of infomorphisms between $\CQ$-distributors. An infomorphism between $\CQ$-distributors is what a Chu transform between Chu spaces \cite{Barr1991,Pratt1995}. The terminology "infomorphism" is from \cite{Barwise1997,Ganter2007}.

\begin{defn}
Given $\CQ$-distributors $\phi:\bbA\olra\bbB$ and $\psi:\bbA'\olra\bbB'$, an infomorphism $(F,G):\phi\lra\psi$ is a pair of $\CQ$-functors $F:\bbA\lra\bbA'$ and $G:\bbB'\lra\bbB$
such that $G^\natural \circ\phi=\psi\circ F_\natural$, or equivalently, $\phi(-,G-)=\psi(F-,-)$.
$$\bfig
\square[\bbA`\bbB`\bbA'`\bbB';\phi`F_\natural `G^\natural`\psi]
\place(250,0)[\circ]
\place(250,500)[\circ]
\place(0,250)[\circ]
\place(500,250)[\circ]
\efig$$
\end{defn}

An adjunction $F\dv G:\bbA\rhu\bbB$ in $\CQ$-{\bf Cat} is exactly an infomorphism from the identity $\CQ$-distributor on $\bbA$ to the identity $\CQ$-distributor on $\bbB$. Thus, infomorphisms are an extension of adjoint $\CQ$-functors.

$\CQ$-distributors and infomorphisms constitute a category $\CQ$-{\bf Info}. The primary aim of this paper is to show that the constructions of Isbell adjunctions and Kan adjunctions are functors defined on $\CQ$-{\bf Info}.

\begin{prop} \label{Y_functor_Cat_Info}
Let $F:\bbA\lra\bbB$ be a $\CQ$-functor, then
$$(F,F^{\la}):((\sY_{\bbA})_{\natural}:\bbA\olra\PA)\lra((\sY_{\bbB})_{\natural}:\bbB\olra\PB)$$
is an infomorphism.
\end{prop}

\begin{proof}
For all $x\in\bbA_0$ and $\lam\in\PB$,
\begin{align*}
(\sY_{\bbA})_{\natural}(x,F^{\la}(\lam))&= \PA(\sY_\bbA(x),F^{\la}(\lam))\\
&=F^{\la}(\lam)(x)&\text{(by Yoneda lemma)}\\
&=\lam(Fx)&\text{(by Equation (\ref{F_la_lam_Fx}))}\\
&= \PB(\sY_\bbB(Fx),\lam)&\text{(by Yoneda lemma)}\\
&=(\sY_{\bbB})_{\natural}(Fx,\lam).
\end{align*}
Hence the conclusion holds.
\end{proof}

The above proposition gives rise to a fully faithful functor $\BY:\CQ\text{-}{\bf Cat}\lra\CQ\text{-}{\bf Info}$ that sends each $\CQ$-category $\bbA$ to the graph $(\sY_{\bbA})_{\natural}$ of the Yoneda embedding.

\begin{prop} \label{Y_U_adjunction}
$\BY:\CQ\text{-}{\bf Cat}\lra\CQ\text{-}{\bf Info}$ is a left adjoint of the forgetful functor $\BU:\CQ\text{-}{\bf Info}\lra\CQ\text{-}{\bf Cat}$ that sends an infomorphism
$$(F,G):(\phi:\bbA\olra\bbB)\lra(\psi:\bbA'\olra\bbB')$$
to the $\CQ$-functor $F:\bbA\lra\bbA'$.
\end{prop}

\begin{proof}
It is clear that $\BU\circ\BY={\bf id}_{\CQ\text{-}{\bf Cat}}$, the identity functor on $\CQ$-{\bf Cat}. Thus $\{1_{\bbA}\}$ is a natural transformation from ${\bf id}_{\CQ\text{-}{\bf Cat}}$ to $\BU\circ\BY$. It remains to show that for each $\CQ$-category $\bbA$, $\CQ$-distributor $\psi:\bbA'\lra\bbB'$ and $\CQ$-functor $H:\bbA\lra\bbA'$, there is a unique infomorphism
$$(F,G):\BY(\bbA)\lra(\psi:\bbA'\olra\bbB')$$
such that the diagram
$$\bfig
\qtriangle<700,500>[\bbA`\BU\circ\BY(\bbA)`\bbA';1_{\bbA}`H`\BU(F,G)]
\efig$$
is commutative. By definition, $\BY(\bbA)$ is the graph $(\sY_{\bbA})_{\natural}:\bbA\olra\PA$ and $\BU(F,G)=F$. Thus, we only need to show that there is a unique $\CQ$-functor $G:\bbB'\lra\PA$ such that
$$(H,G):((\sY_{\bbA})_{\natural}:\bbA\olra\PA)\lra(\psi:\bbA'\olra\bbB')$$
is an infomorphism.

Let $G=H^{\la}\circ\lapsi:\bbB'\lra\PA$, where $\lapsi:\bbB'\lra\PA'$ is the $\CQ$-functor assigning each  $y'\in\bbB'_0$ to $\psi(-,y')$ in $\PA$. Then $$(H,G):((\sY_{\bbA})_{\natural}:\bbA\olra\PA) \lra(\psi:\bbA'\olra\bbB')$$
is an infomorphism since
$$(\sY_{\bbA})_{\natural}(x,Gy')=(Gy')(x)=H^{\la}\circ\lapsi(y')(x)=\lapsi(y')(Hx)=\psi(Hx,y')$$
for all $x\in\bbA_0$ and $y'\in\bbB'_0$. This proves the existence of $G$.

To see the uniqueness of $G$, suppose that $G':\bbB'\lra\PA$ is another $\CQ$-functor such that
$$(H,G'):((\sY_{\bbA})_{\natural}:\bbA\olra\PA)\lra(\psi:\bbA'\olra\bbB')$$
is an infomorphism. Then for all $x\in\bbA_0$ and $y'\in\bbB'_0$,
$$(G'y')(x)=(\sY_{\bbA})_{\natural}(x,G'y')=\psi(Hx,y')=\lapsi(y')(Hx)=H^{\la}\circ\lapsi(y')(x)=(Gy')(x),$$
hence $G'=G$.
\end{proof}

Similar to Proposition \ref{Y_functor_Cat_Info}, one can check that sending a $\CQ$-functor $F:\bbA\lra\bbB$ to the infomorphism
$$(F^{\La},F):((\sYd_{\bbB})^{\natural}:\PdB\olra\bbB)\lra((\sYd_{\bbA})^{\natural}:\PdA\olra\bbA)$$
induces a fully faithful functor $\BY^{\dag}:\CQ\text{-}{\bf Cat}\lra(\CQ\text{-}{\bf Info})^{\rm op}$.

\begin{prop} \label{Y_U_adjunction_contravariant}
$\BY^{\dag}:\CQ\text{-}{\bf Cat}\lra(\CQ\text{-}{\bf Info})^{\rm op}$ is a left adjoint of the contravariant forgetful functor $(\CQ\text{-}{\bf Info})^{\rm op}\lra\CQ\text{-}{\bf Cat}$ that sends each infomorphism
$$(F,G):(\phi:\bbA\olra\bbB)\lra(\psi:\bbA'\olra\bbB')$$
to the $\CQ$-functor $G:\bbB'\lra\bbB$.
\end{prop}

\begin{proof}
Similar to Proposition \ref{Y_U_adjunction}.
\end{proof}

\section{$\CQ$-closure spaces} \label{closure_space}

\begin{defn}
Let $\bbA$ be a $\CQ$-category.
\begin{itemize}
\item[\rm (1)] An isomorphism-closed  $\CQ$-subcategory $\bbB$ of $\bbA$ is a {\it $\CQ$-closure system} (resp. {\it $\CQ$-interior system}) of $\bbA$ if the inclusion $\CQ$-functor $I:\bbB\lra\bbA$ is a right (resp. left) adjoint.
\item[\rm (2)] A $\CQ$-functor $F:\bbA\lra\bbA$ is a {\it $\CQ$-closure operator} (resp. {\it $\CQ$-interior operator}) on $\bbA$ if $1_{\bbA}\leq F$ (resp. $F\leq 1_{\bbA}$) and $F^2\cong F$.
\end{itemize}
\end{defn}

\begin{exmp} \label{adjunction_closure_interior}
Let $F\dv G:\bbA\rhu\bbB$ be an adjunction in $\CQ$-{\bf Cat}. Then $G\circ F:\bbA\lra\bbA$ is a $\CQ$-closure operator and $F\circ G:\bbB\lra\bbB$ is a $\CQ$-interior operator.
\end{exmp}

\begin{prop} \label{closure_interior_system_operator}
Let $\bbA$ be a $\CQ$-category, $\bbB$ an isomorphism-closed $\CQ$-subcategory of $\bbA$. The following conditions are equivalent:
\begin{itemize}
\item[\rm (1)] $\bbB$ is a $\CQ$-closure system (resp. $\CQ$-interior system) of $\bbA$.
\item[\rm (2)] There is a $\CQ$-closure operator (resp. $\CQ$-interior operator) $F:\bbA\lra\bbA$ such that $\bbB_0=\{x\in\bbA_0:Fx\cong x\}$.
\end{itemize}
\end{prop}

\begin{proof}
$(1)\Rightarrow(2)$: If the inclusion $\CQ$-functor $I:\bbB\lra\bbA$ has a left adjoint $G:\bbA\lra\bbB$, let $F=I\circ G$, then $F:\bbA\lra\bbA$ is a $\CQ$-closure operator. Since $Fx=Gx\in\bbB_0$ for all $x\in\bbA_0$ and $\bbB$ is isomorphism-closed, it is clear that $\{x\in\bbA_0:Fx\cong x\}\subseteq\bbB_0$. Conversely, for all $x\in\bbB_0$,
$$\bbB(Fx,x)=\bbB(Gx,x)=\bbA(x,Ix)=\bbA(x,x)\geq 1_{tx},$$
and $\bbB(x,Fx)\geq 1_{tx}$ holds trivially, hence $x\cong Fx$, as required.

$(2)\Rightarrow(1)$: We show that the inclusion $\CQ$-functor $I:\bbB\lra\bbA$ is a right adjoint. View $F$ as a $\CQ$-functor from $\bbA$ to $\bbB$, then $1_{\bbA}\leq I\circ F$. Since $F^2\cong F$, it follows that $F\circ I\cong 1_{\bbB}$. Thus $F\dv I:\bbA\rhu\bbB$, as required.
\end{proof}

\begin{rem}
For a $\CQ$-category $\bbA$, a $\CQ$-closure operator (resp. $\CQ$-interior operator) $F:\bbA\lra\bbA$ is exactly a monad (resp. comonad) \cite{Lane1998} on $\bbA$. The above proposition states that a $\CQ$-closure system (resp. $\CQ$-interior system) of $\bbA$ is exactly the category of algebras (resp. coalgebras) for a monad (resp. comonad) on $\bbA$. The terminology "$\CQ$-closure operator" (resp. "$\CQ$-interior operator") comes from its similarity to closure (resp. interior) operators in topology.
 \end{rem}

\begin{prop} \label{closure_system_complete}
Each $\CQ$-closure system (resp. $\CQ$-interior system) of a complete $\CQ$-category is itself a complete $\CQ$-category.
\end{prop}

\begin{proof}
Let $\bbB$ be a $\CQ$-closure system of a complete $\CQ$-category $\bbA$. By Proposition \ref{closure_interior_system_operator}, there is a $\CQ$-closure operator $F:\bbA\lra\bbA$ such that $\bbB_0=\{x\in\bbA_0:Fx\cong x\}$. View $F$ as a $\CQ$-functor from $\bbA$ to $\bbB$, then $F$ is essentially surjective and $F\dv I:\bbA\rhu\bbB$, where $I$ is the inclusion $\CQ$-functor. For all $\mu\in\PB$,
\begin{align*}
F({\sup}_{\bbA}I^{\ra}(\mu))&={\sup}_{\bbB}F^{\ra}\circ I^{\ra}(\mu)&(\text{by Corollary \ref{left_adjoint_preserves_sup}})\\
&={\sup}_{\bbB}\mu\circ I^{\natural}\circ F^{\natural}&(\text{by the definition\ of}\ F^{\ra}\ \text{and}\ I^{\ra})\\
&={\sup}_{\bbB}\mu\circ F_{\natural}\circ F^{\natural}\circ I^{\natural}\circ F^{\natural}&(\text{by Proposition \ref{fully_faithful_graph_cograph}(2)})\\
&={\sup}_{\bbB}\mu\circ F_{\natural}\circ (F\circ I\circ F)^{\natural}\\
&={\sup}_{\bbB}\mu\circ F_{\natural}\circ F^{\natural}&(\text{since}\ F\dv I:\bbA\rhu\bbB)\\
&={\sup}_{\bbB}\mu.&(\text{by Proposition \ref{fully_faithful_graph_cograph}(2)})
\end{align*}
Then it follows from Proposition \ref{complete_cocomplete_equivalent} that $F(\bbA)$ is a complete $\CQ$-category.
\end{proof}

\begin{prop} \label{closure_system_cotensor_meet}
Let $\bbA$ be a complete $\CQ$-category with tensor $\otimes$ and cotensor $\ \rat\ $, $\bbB$ an isomorphism-closed $\CQ$-subcategory of $\bbA$. Then $\bbB$ is a $\CQ$-closure system (resp. $\CQ$-interior system) of $\bbA$ if and only if
\begin{itemize}
\item[\rm (1)] for every subset $\{x_i\}\subseteq\bbB_0$ of the same type $X$, the meet $\displaystyle\bw\limits_i x_i$ (resp. the join $\displaystyle\bv\limits_i x_i$) in $\bbA_X$ belongs to $\bbB_0$.
\item[\rm (2)] for each $x\in\bbB_0$ and $f\in\CPd(tx)$ (resp. $f\in\CP(tx)$), the cotensor $f\rat x$ (resp. the tensor $f\otimes x$) in $\bbA$ belongs to $\bbB_0$.
\end{itemize}
\end{prop}

\begin{proof}
Follows immediately from Proposition \ref{F_la_ra_condition}.
\end{proof}

An immediate consequence of Proposition \ref{closure_system_cotensor_meet} is that the infimum (resp. supremum) in a $\CQ$-closure system (resp. $\CQ$-interior system) $\bbB$ of a complete $\CQ$-category $\bbA$ can be calculated as
\begin{equation} \label{closure_system_infimum}
{\inf}_{\bbB}\lam=\bw_{b\in\bbB_0}(\lam(b)\rat b),\quad\Big(\text{resp.}\ {\sup}_{\bbB}\mu=\bv_{b\in\bbB_0}(\mu(b)\otimes b)\Big)
\end{equation}
for $\lam\in\PdB$ (resp. $\mu\in\PB$), where the cotensors and meets (resp. tensors and joins) are calculated in $\bbA$.

\begin{defn}
A {\it $\CQ$-closure space} is a pair $(\bbA,C)$ that consists of a $\CQ$-category $\bbA$ and a $\CQ$-closure operator $C:\PA\lra\PA$. A {\it continuous $\CQ$-functor} $F:(\bbA,C)\lra(\bbB,D)$ between $\CQ$-closure spaces is a $\CQ$-functor $F:\bbA\lra\bbB$ such that $F^{\ra}\circ C\leq D\circ F^{\ra}$.
 The category of $\CQ$-closure spaces and continuous $\CQ$-functors is denoted by $\CQ$-{\bf Cls}.
\end{defn}

\begin{rem}
If $C, D$ are viewed as monads on $\CP\bbA, \CP\bbB$ respectively, then a $\CQ$-functor $F:(\bbA,C)\lra(\bbB,D)$ between $\CQ$-closure spaces is continuous if and only if $F^\ra:\CP\bbA\lra\CP\bbB$ is a lax map of monads from $C$ to $D$ in the sense of \cite{Leinster2004}.
\end{rem}

Note that for a $\CQ$-closure space $(\bbA,C)$, the $\CQ$-closure operator $C$ is idempotent since $\PA$ is skeletal. Let $C(\PA)$ denote the $\CQ$-subcategory of $\PA$ consisting of the fixed points of $C$. Since $\PA$ is a complete $\CQ$-category,  $C(\PA)$ is also a complete $\CQ$-category. A contravariant presheaf $\bbA\olra*_X$ is said to be {\it closed} in the $\CQ$-closure space $(\bbA,C)$ if it belongs to $C(\PA)$.
The following lemma states that continuous $\CQ$-functors behave in a manner similar to the continuous maps between topological spaces: the inverse image of a closed contravariant presheaf is closed.
\begin{lem}
A $\CQ$-functor $F:(\bbA,C)\lra(\bbB,D)$ between $\CQ$-closure spaces is continuous if and only if $F^{\la}(\lam)\in C(\PA)$ whenever $\lam\in D(\PB)$.
\end{lem}

\begin{proof}
It suffices to show that $F^{\ra}\circ C\leq D\circ F^{\ra}$ if and only if $C\circ F^{\la}\circ D\leq F^{\la}\circ D$.

Suppose $F^{\ra}\circ C\leq D\circ F^{\ra}$, then
$$F^{\ra}\circ C\circ F^{\la}\circ D\leq D\circ F^{\ra}\circ F^{\la}\circ D\leq D\circ D=D,$$
and consequently $C\circ F^{\la}\circ D\leq F^{\la}\circ D$.

Conversely, suppose $C\circ F^{\la}\circ D\leq F^{\la}\circ D$, then
$$C\leq C\circ F^{\la}\circ F^{\ra}\leq C\circ F^{\la}\circ D\circ F^{\ra}\leq F^{\la}\circ D\circ F^{\ra},$$
and consequently $F^{\ra}\circ C\leq D\circ F^{\ra}$.
\end{proof}

Thus a continuous $\CQ$-functor $F:(\bbA,C)\lra(\bbB,D)$ between $\CQ$-closure spaces induces a pair of $\CQ$-functors
$$F^{\triangleright}=D\circ F^{\ra}:C(\PA)\lra D(\PB)\quad\text{and}\quad F^{\triangleleft}=F^{\la}:D(\PB)\lra C(\PA).$$

\begin{prop}
If $F:(\bbA,C)\lra(\bbB,D)$ is a continuous $\CQ$-functor between $\CQ$-closure spaces, then $F^{\triangleright}\dv F^{\triangleleft}:C(\PA)\rhu D(\PB)$.
\end{prop}

\begin{proof}
It is sufficient to check that
$$\PB(D\circ F^{\ra}(\mu),\lam)=\PB(F^{\ra}(\mu),\lam)$$
for all $\mu\in C(\PA)$ and $\lam\in D(\PB)$ since it holds that $\PA(\mu,F^\la(\lambda))=\PB(F^\ra(\mu),\lambda)$. Indeed, since $D$ is  a $\CQ$-closure operator,
\begin{align*}
\PB(F^{\ra}(\mu),\lam)&\leq\PB(D\circ F^{\ra}(\mu),D(\lam))\\
&=\PB(D\circ F^{\ra}(\mu),\lam)\\
&=\lam\lda(D\circ F^{\ra}(\mu))\\
&\leq\lam\lda F^{\ra}(\mu)\\
&=\PB(F^{\ra}(\mu),\lam),
\end{align*} hence $\PB(D\circ F^{\ra}(\mu),\lam)=\PB(F^{\ra}(\mu),\lam)$.
\end{proof}

Skeletal complete $\CQ$-categories constitute a full subcategory of $\CQ$-{\bf CCat} and we denote it by $(\CQ\text{-}{\bf CCat})_{\skel}$. The above proposition gives rise to a functor
$$\CT:\CQ{\text -}{\bf Cls}\lra(\CQ\text{-}{\bf CCat})_{\skel}$$
that maps a continuous $\CQ$-functor $F:(\bbA,C)\lra(\bbB,D)$ to a left adjoint $\CQ$-functor $F^{\triangleright}:C(\PA)\lra D(\PB)$ between skeletal complete $\CQ$-categories.

For each complete $\CQ$-category $\bbA$, it follows from Theorem \ref{complete_cocomplete_equivalent} and Example \ref{adjunction_closure_interior} that $C_{\bbA}=\sY\circ\sup:\PA\lra\PA$ is a $\CQ$-closure operator, hence $(\bbA,C_{\bbA})$ is a $\CQ$-closure space.

\begin{prop}
If $F:\bbA\lra\bbB$ is a left adjoint $\CQ$-functor between complete $\CQ$-categories, then $F:(\bbA,C_{\bbA})\lra(\bbB,C_{\bbB})$ is a continuous $\CQ$-functor.
\end{prop}

\begin{proof}
For all $\mu\in\PA$,
\begin{align*}
F^{\ra}\circ C_{\bbA}(\mu)&=C_{\bbA}(\mu)\circ F^{\natural}\\
&=\bbA(-,{\sup}_{\bbA}\mu)\circ F^{\natural}\\
&\leq\bbB(F-,F({\sup}_{\bbA}\mu))\circ F^{\natural}\\
&=F_{\natural}(-,F({\sup}_{\bbA}\mu))\circ F^{\natural}\\
&\leq\bbB(-,F({\sup}_{\bbA}\mu))&(\text{since}\ F_{\natural}\dv F^{\natural}:\bbA\rhu\bbB\ \text{in}\ \CQ\text{-}{\bf Dist})\\
&=\bbB(-,{\sup}_{\bbB}F^{\ra}(\mu))&(\text{by Corollary \ref{left_adjoint_preserves_sup}})\\
&=C_{\bbB}\circ F^{\ra}(\mu).
\end{align*}
Hence $F:(\bbA,C_{\bbA})\lra(\bbB,C_{\bbB})$ is continuous.
\end{proof}

The above proposition gives a functor $\CD:(\CQ\text{-}{\bf CCat})_{\skel}\lra\CQ\text{-}{\bf Cls}$.

For each $\CQ$-category $\bbA$, it is clear that $C_{\bbA}(\PA)=\{\sY_{\bbA} a\mid a\in\bbA_0\}$. So, for a skeletal $\CQ$-category $\bbA$, if we identify $\bbA$ with the $\CQ$-subcategory $C_{\bbA}(\PA)$ of $\PA$, then the functor $\CT:\CQ\text{-}{\bf Cls}\lra (\CQ\text{-}{\bf CCat})_{\skel}$ is a left inverse of $\CD:(\CQ\text{-}{\bf CCat})_{\skel}\lra\CQ\text{-}{\bf Cls}$ since $\CT\circ\CD(\bbA)=C_{\bbA}(\PA)$.

\begin{thm} \label{T_D_adjunction}
$\CT:\CQ\text{-}{\bf Cls}\lra (\CQ\text{-}{\bf CCat})_{\skel}$ is a left inverse and left adjoint of $\CD:(\CQ\text{-}{\bf CCat})_{\skel} \lra \CQ\text{-}{\bf Cls}$.
\end{thm}

\begin{proof}
It remains to show that $\CT$ is a left adjoint of $\CD$ . Given a $\CQ$-closure space $(\bbA,C)$, denote $C(\PA)$ by $\bbX$, then $\CD\circ\CT(\bbA,C)=(\bbX,C_{\bbX})$. Let $\eta_{(\bbA,C)}=C\circ\sY_{\bbA}:\bbA\lra\bbX$. We show that $\eta=\{\eta_{(\bbA,C)}\}$ is a natural transformation from the identity functor  to $\CD\circ\CT$ and it is the unit of the desired adjunction.

{\bf Step 1}. $\eta_{(\bbA,C)}:(\bbA,C)\lra (\bbX,C_{\bbX})$ is a continuous $\CQ$-functor, i.e. $\eta_{(\bbA,C)}^{\ra}\circ C\leq C_{\bbX}\circ\eta_{(\bbA,C)}^{\ra}$.

Firstly, we show that $C(\mu)=\sup_{\bbX}\circ\eta_{(\bbA,C)}^{\ra}(\mu)$ for all $\mu\in\PA$. Consider the diagram:
$$\bfig
\Square[\CP(\PA)`\PA`\PX`\bbX;\sup_{\PA}`C^\ra`C`\sup_{\bbX}]
\morphism(-600,500)<600,0>[\PA`\CP(\PA);\sY_{\bbA}^{\ra}]
\morphism(-600,500)|l|<600,-500>[\PA`\CP\bbX;\eta_{(\bbA,C)}^\ra]
\efig$$
The commutativity of the left triangle follows from $\eta_{(\bbA,C)}=C\circ \sY_{\bbA}$. Since $C:\PA\lra\bbX$ is a left adjoint in $\CQ$-{\bf Cat} (obtained in the proof of Proposition \ref{closure_interior_system_operator}), it preserves supremum (Corollary \ref{left_adjoint_preserves_sup}), thus the right square commutes. The whole diagram is then commutative. For each $\mu\in\PA$, we have that
\begin{equation} \label{mu_sup_ymu}
\mu=\mu\circ\bbA=\mu\circ\sY_{\bbA}^{\natural}\circ(\sY_{\bbA})_{\natural}=\sY_{\bbA}^\ra(\mu)\circ(\sY_{\bbA})_{\natural}={\sup}_{\PA}\circ\sY_{\bbA}^\ra(\mu),
\end{equation}
where the second equality comes from the fact that the Yoneda embedding $\sY_{\bbA}$ is fully faithful and Proposition \ref{fully_faithful_graph_cograph}(1), while the last equality comes from Example \ref{PX_PA_complete}. Consequently,
$$C(\mu)=C\circ{\sup}_{\PA}\circ\sY_{\bbA}^\ra(\mu)={\sup}_{\bbX}\circ\eta_{(\bbA,C)}^{\ra}(\mu)$$
for all $\mu\in\PA$.

Secondly, we show that $\eta_{(\bbA,C)}^{\ra}(\mu)\leq\sY_{\bbX}(\mu)=\bbX(-,\mu)$ for each $\mu\in\bbX$. Indeed,
\begin{align*}
\eta_{(\bbA,C)}^{\ra}(\mu)&=\mu\circ\eta_{(\bbA,C)}^{\natural}\\
&=\mu\circ(C\circ\sY_{\bbA})^{\natural}\\
&=\PA(\sY_\bbA-,\mu)\circ\sY_{\bbA}^{\natural}\circ C^{\natural}&(\text{by Yoneda lemma})\\
&=(\sY_{\bbA})_{\natural}(-,\mu)\circ\sY_{\bbA}^{\natural}\circ C^{\natural}& \\
&\leq\PA(-,\mu)\circ C^{\natural}&(\text{since}\ (\sY_{\bbA})_{\natural}\dv\sY_{\bbA}^{\natural}:\bbA\rhu\PA\ \text{in}\ \CQ\text{-}{\bf Dist})\\
&\leq\bbX(C-,\mu)\circ C^{\natural}&(\text{since}\ C\ \text{is a}\ \CQ\text{-functor and}\ C(\mu)=\mu)\\
&=C_{\natural}(-,\mu)\circ C^{\natural}\\
&\leq\bbX(-,\mu).&(\text{since}\ C_{\natural}\dv C^{\natural}:\PA\rhu\bbX\ \text{in}\ \CQ\text{-}{\bf Dist})
\end{align*}

Therefore, for all $\mu\in\PA$,
$$\eta_{(\bbA,C)}^{\ra}\circ C(\mu)\leq\sY_{\bbX}\circ{\sup}_{\bbX}\circ\eta_{(\bbA,C)}^{\ra}(\mu)=C_{\bbX}\circ\eta_{(\bbA,C)}^{\ra}(\mu),$$
as desired.

{\bf Step 2}. $\eta=\{\eta_{(\bbA,C)}\}$ is a natural transformation. Let $F:(\bbA,C)\lra(\bbB,D)$ be a continuous $\CQ$-functor, we must show that
$$D\circ\sY_{\bbB}\circ F=\eta_{(\bbB,D)}\circ F=\CD\circ\CT\circ F\circ\eta_{(\bbA,C)}=D\circ F^{\ra}\circ C\circ\sY_{\bbA}.$$

Firstly, we show that
\begin{equation} \label{Yoneda_natural}
\sY_{\bbB}\circ F=F^{\ra}\circ\sY_{\bbA}
\end{equation}
for each $\CQ$-functor $F:\bbA\lra\bbB$. Indeed, for all $x\in\bbA_0$,
\begin{align*}
\sY_{\bbB}\circ Fx&=F^{\natural}(-,x)&(\text{by the definition of}\ F^{\natural})\\
&=(\bbA\circ F^{\natural})(-,x)\\
&=\bbA(-,x)\circ F^{\natural}&(\text{by Remark \ref{distributor_notion}})\\
&=F^{\ra}\circ\sY_{\bbA}x.&(\text{by the definition of}\ F^{\ra})
\end{align*}

Secondly, since $C$ is a $\CQ$-closure operator,
$$\sY_{\bbB}\circ F=F^{\ra}\circ\sY_{\bbA}\leq F^{\ra}\circ C\circ\sY_{\bbA},$$
and consequently $D\circ\sY_{\bbB}\circ F\leq D\circ F^{\ra}\circ C\circ\sY_{\bbA}$.

Thirdly, the continuity of $F$ leads to
$$F^{\ra}\circ C\circ\sY_{\bbA}\leq D\circ F^{\ra}\circ\sY_{\bbA}=D\circ\sY_{\bbB}\circ F,$$
hence $D\circ F^{\ra}\circ C\circ\sY_{\bbA}\leq D\circ\sY_{\bbB}\circ F$.

{\bf Step 3}. $\eta_{(\bbA,C)}:(\bbA,C)\lra (\bbX,C_{\bbX})$ is universal in the sense that for any skeletal complete $\CQ$-category $\bbB$ and continuous $\CQ$-functor $F:(\bbA,C)\lra (\bbB,C_{\bbB})$, there exists a unique left adjoint $\CQ$-functor $\overline{F}:\bbX\lra\bbB$ that makes the following diagram commute:
\begin{equation} \label{eta_universal}
\bfig
\qtriangle<600,500>[(\bbA,C)`(\bbX,C_{\bbX})`(\bbB,C_{\bbB});\eta_{(\bbA,C)}`F`\overline{F}]
\efig
\end{equation}

{\bf Existence.} Let $\overline{F}=\sup_{\bbB}\circ F^\ra:\bbX\lra\bbB$ be the following composition of $\CQ$-functors
$$\bbX\hookrightarrow\PA\to^{F^\ra}\PB\to^{\sup_{\bbB}}\bbB. $$

First,  $\overline{F}:\bbX\lra\bbB$ is a left adjoint in $\CQ$-{\bf Cat}. Indeed, $\overline{F}$ has a right adjoint $G:\bbB\lra\bbX$ given by $G=F^{\triangleleft}\circ\sY_{\bbB}$. $G$ is well-defined since $\sY_{\bbB}b$ is a closed  in $(\bbB,C_{\bbB})$ for each $b\in\bbB_0$. For all $\mu\in\bbX_0$ and $y\in\bbB_0$, it holds that
\begin{align*}
\bbB(\overline{F}(\mu),y)&=\bbB(-,y)\lda F^{\ra}(\mu)\\
&=\bbB(-,y)\lda(\mu\circ F^{\natural})\\
&=(\bbB(-,y)\circ F_\natural)\lda\mu&(\text{by Proposition \ref{graph_cograph_implication}(2)})\\
&=F_\natural (-,y)\lda\mu&(\text{by Remark \ref{distributor_notion}})\\
&=\PA(\mu,F^{\triangleleft}\circ\sY_{\bbB}y)&(\text{by the definition of}\ F_{\natural}\ \text{and}\ F^{\triangleleft})\\
&=\bbX(\mu,Gy),
\end{align*}
hence  $\overline{F}$ is a left adjoint of $G$.

Second,  $F=\overline{F}\circ\eta_{(\bbA,C)}$. Note that for all $x\in\bbA_0$,
\begin{align*}
\bbB(Fx,-)&=F_{\natural}(x,-)\\
&=(\bbB\lda F^{\natural})(x,-)& (\text{by Proposition \ref{graph_cograph_implication}(1)})\\
&=\bbB\lda F^{\natural}(-,x)\\
&=\bbB\lda(\sY_{\bbB}\circ Fx)\\
&=\bbB\lda(F^{\ra}\circ\sY_{\bbA}x),&(\text{by Equation (\ref{Yoneda_natural})})
\end{align*}
thus $F=\sup_{\bbB}\circ F^{\ra}\circ\sY_{\bbA}$. Consequently
\begin{align*}
\overline{F}\circ\eta_{(\bbA,C)}&={\sup}_{\bbB}\circ F^{\ra}\circ C\circ\sY_{\bbA}\\
&\leq{\sup}_{\bbB}\circ C_{\bbB}\circ F^{\ra}\circ\sY_{\bbA}&(\text{since}\ F\ \text{is continuous})\\
&={\sup}_{\bbB}\circ \sY_{\bbB}\circ{\sup}_{\bbB}\circ F^{\ra}\circ\sY_{\bbA}\\
&={\sup}_{\bbB}\circ F^{\ra}\circ\sY_{\bbA}&(\text{since}\ {\sup}_{\bbB}\dv\sY_{\bbB}:\PB\rhu\bbB)\\
&=F.
\end{align*}
Conversely, since $C$ is a $\CQ$-closure operator, it is clear that
$$F={\sup}_{\bbB}\circ F^{\ra}\circ\sY_{\bbA}\leq{\sup}_{\bbB}\circ F^{\ra}\circ C\circ\sY_{\bbA}=\overline{F}\circ\eta_{(\bbA,C)},$$
hence $F\cong\overline{F}\circ\eta_{(\bbA,C)}$, and consequently $F=\overline{F}\circ\eta_{(\bbA,C)}$ since $\bbB$ is skeletal.

{\bf Uniqueness.} Suppose $H:\bbX\lra\bbB$ is another left adjoint $\CQ$-functor that makes Diagram (\ref{eta_universal}) commute. For each $\mu\in\bbX$, since $C:\PA\lra\bbX$ is a left adjoint in $\CQ$-{\bf Cat}, we have
$$\mu=C(\mu)=C(\mu\circ\bbA)=C\Big(\bv_{x\in\bbA_0}\mu(x)\circ\sY_{\bbA}x\Big)=\bv_{x\in\bbA_0}\mu(x)\otimes_{\bbX}C(\sY_{\bbA}x),$$
where the last equality follows from Example \ref{PA_tensor} and Proposition \ref{F_la_ra_condition}. It follows that
\begin{align*}
H(\mu)&=H\Big(\bv_{x\in\bbA_0}\mu(x)\otimes_{\bbX}C(\sY_{\bbA}x)\Big) \\
&=\bv_{x\in\bbA_0}\mu(x)\otimes_{\bbB}(H\circ\eta_{(\bbA,C)}(x))&\text{(by Proposition \ref{F_la_ra_condition})}\\
&=\bv_{x\in\bbA_0}\mu(x)\otimes_{\bbB} Fx.
\end{align*}
Consequently,
\begin{align*}
\bbB(H(\mu),-)&=\bbB\Big(\bv_{x\in\bbA_0}\mu(x)\otimes_{\bbB} Fx,-\Big)\\
&=\bw_{x\in\bbA_0}\Big(\bbB(Fx,-)\lda\mu(x)\Big)\\
&=F_\natural \lda\mu\\
&=\bbB\lda(\mu\circ F^{\natural})&(\text{by Proposition \ref{graph_cograph_implication}(2)})\\
&=\bbB\lda F^{\ra}(\mu).
\end{align*}
Since $\bbB$ is skeletal, it follows that $H(\mu)=\sup_{\bbB}\circ F^{\ra}(\mu)$. Therefore, $H=\sup_{\bbB}\circ F^{\ra}=\overline{F}$.
\end{proof}

\section{Isbell adjunctions} \label{Isbell_adjunction}

Given a $\CQ$-distributor $\phi:\bbA\olra\bbB$, define a pair of $\CQ$-functors
$$\uphi:\PA\lra\PdB\quad\text{and}\quad\dphi:\PdB\lra\PA$$
by
$$\uphi(\mu)=\phi\lda\mu\quad\text{and}\quad\dphi(\lam)=\lam\rda\phi.$$
It should be warned that $\uphi$ and $\dphi$ are both contravariant with respect to local orders in $\CQ$-{\bf Dist} by Remark \ref{PdA_QDist_order}, i.e.,
\begin{equation} \label{uphi_contravariant}
\forall\mu_1,\mu_2\in\PA,\mu_1\leq\mu_2\Longrightarrow\uphi(\mu_2)\leq\uphi(\mu_1)
\end{equation}
and
\begin{equation} \label{dphi_contravariant}
\forall\lam_1,\lam_2\in\PdB,\lam_1\leq\lam_2\Longrightarrow\dphi(\lam_2)\leq\dphi(\lam_1).
\end{equation}

\begin{prop} \label{uphi-dphi-adjunction}
$\uphi\dv\dphi:\PA\rhu\PdB$ in $\CQ$-{\bf Cat}.
\end{prop}

\begin{proof}
For all $\mu\in\PA$ and $\lam\in\PdB$,
\begin{align*}
\PdB(\uphi(\mu),\lam)&=\lam\rda\uphi(\mu)\\
&=\lam\rda(\phi\lda\mu)\\
&=(\lam\rda\phi)\lda\mu\\
&=\dphi(\lam)\lda\mu\\
&=\PA(\mu,\dphi(\lam)).
\end{align*} Hence the conclusion holds.
\end{proof}

Letting $\bbB=\bbA$ and $\phi=\bbA$ in Proposition \ref{uphi-dphi-adjunction} gives the following

\begin{cor} {\rm\cite{Stubbe_2005}} \label{ub_lb_adjunction}
$\bbA\swarrow(-)\dv(-)\searrow\bbA:\PA\rhu\PdA$.
\end{cor}

The adjunction in Corollary \ref{ub_lb_adjunction} is known as the Isbell adjunction in category theory.  So,  the adjunction  $\uphi\dv\dphi:\PA\rhu\PdB$ is a generalization of the Isbell adjunction. As we shall see, all adjunctions between $\PA$ and $\PdB$ are of this form, and will  be called Isbell adjunctions by abuse of language.

Each $\CQ$-functor $F:\bbA\lra\PdB$ corresponds to a $\CQ$-distributor $\hF:\bbA\olra\bbB$ given by $\hF(x,y)=F(x)(y)$ for all $x\in\bbA_0$ and $y\in\bbB_0$, and each $\CQ$-functor $G:\bbB\lra\PA$ corresponds to a $\CQ$-distributor $\hG:\bbA\olra\bbB$ given by $\hG(x,y)=G(y)(x)$ for all $x\in\bbA_0$ and $y\in\bbB_0$.

\begin{prop} \label{uphi_dphi_Yoneda}
Let $\phi:\bbA\olra\bbB$ be a $\CQ$-distributor, then $\ulc\uphi\circ\sY_{\bbA}\urc=\phi=\ulc\dphi\circ\sYd_{\bbB}\urc$.
\end{prop}

\begin{proof}
For all $x\in\bbA_0$ and $y\in\bbB_0$,
\begin{align*}
\ulc\uphi\circ\sY_{\bbA}\urc(x,y)&=(\uphi\circ\sY_{\bbA}x)(y)\\
&=(\phi\lda(\sY_{\bbA}x))(y)\\
&=\phi(-,y)\lda\bbA(-,x)\\
&=\phi(x,y)\\
&=\bbB(y,-)\rda\phi(x,-)\\
&=((\sYd_{\bbB}y)\rda\phi)(x)\\
&=(\dphi\circ\sYd_{\bbB}y)(x)\\
&=\ulc\dphi\circ\sYd_{\bbB}\urc(x,y),
\end{align*} showing that the conclusion holds.
\end{proof}

\begin{thm} \label{Isbell_distributor_bijection}
The correspondence $\phi\mapsto\uphi$ is an isomorphism of posets
$$\CQ\text{-}{\bf Dist}(\bbA,\bbB)\cong\CQ\text{-}{\bf CCat}^{\rm co}(\PA,\PdB),$$
where the "${\rm co}$" means reversing order in the hom-sets.
\end{thm}

\begin{proof}
Let $F:\PA\lra\PdB$ be a left adjoint $\CQ$-functor. We show that the correspondence $F\mapsto\ulc F\circ\sY_{\bbA}\urc$ is an inverse of the correspondence $\phi\mapsto\uphi$, and thus they are both isomorphisms of posets between $\CQ\text{-}{\bf Dist}(\bbA,\bbB)$ and $\CQ\text{-}{\bf CCat}^{\rm co}(\PA,\PdB)$.

Firstly, we show that both of the correspondences are order-preserving. Indeed,
\begin{align*}
&\phi\leq\psi\ \text{in}\ \CQ\text{-}{\bf Dist}(\bbA,\bbB)\\
\iff&\forall\mu\in\PA,\uphi(\mu)=\phi\lda\mu\leq\psi\lda\mu=\upsi(\mu)\ \text{in}\ \CQ\text{-}{\bf Dist}\\
\iff&\forall\mu\in\PA,\uphi(\mu)\geq\upsi(\mu)\ \text{in}\ (\PdB)_0\\
\iff&\uphi\leq\upsi\ \text{in}\ \CQ\text{-}{\bf CCat}^{\rm co}(\PA,\PdB)
\end{align*}
and
\begin{align*}
&F\leq G\ \text{in}\ \CQ\text{-}{\bf CCat}^{\rm co}(\PA,\PdB)\\
\iff&\forall\mu\in\PA,F(\mu)\geq G(\mu)\ \text{in}\ (\PdB)_0\\
\iff&\forall\mu\in\PA,F(\mu)\leq G(\mu)\ \text{in}\ \CQ\text{-}{\bf Dist}(\bbA,\bbB)\\
\Longrightarrow{}&\forall x\in\bbA_0,\ulc F\circ\sY_{\bbA}\urc(x,-)=F\circ\sY_{\bbA}x\leq G\circ\sY_{\bbA}x=\ulc G\circ\sY_{\bbA}\urc(x,-)\ \text{in}\ \CQ\text{-}{\bf Dist}(\bbA,\bbB)\\
\iff&\ulc F\circ\sY_{\bbA}\urc\leq\ulc G\circ\sY_{\bbA}\urc\ \text{in}\ \CQ\text{-}{\bf Dist}(\bbA,\bbB).
\end{align*}

Secondly,  $F=(\ulc F\circ\sY_{\bbA}\urc)_{\ua}$. For all $\mu\in\PA$, since $F$ is a left adjoint in $\CQ$-{\bf Cat}, by Example \ref{PA_tensor} and Proposition \ref{F_la_ra_condition} we have
\begin{align*}
F(\mu)&=F(\mu\circ\bbA)\\
&=F\Big(\bv_{x\in\bbA_0}\mu(x)\circ\sY_{\bbA}x\Big)\\
&=\bw_{x\in\bbA_0}(F\circ\sY_{\bbA}x)\lda\mu(x)\\
&=\ulc F\circ\sY_{\bbA}\urc\lda\mu\\
&=(\ulc F\circ\sY_{\bbA}\urc)_{\ua}(\mu).
\end{align*}

Finally, $\phi=\ulc\uphi\circ\sY_{\bbA}\urc$. This is obtained in Proposition \ref{uphi_dphi_Yoneda}.
\end{proof}

For a $\CQ$-distributor $\phi:\bbA\olra\bbB$, one obtains two $\CQ$-functors $\raphi:\bbA\lra\PdB$ and $\laphi:\bbB\lra\PA$ by letting $\raphi x=\phi(x,-)$ for all $x\in\bbA_0$ and $\laphi y=\phi(-,y)$ for all $y\in\bbB_0$. Stubbe \cite{Stubbe_2005} shows that the maps $\phi\mapsto\raphi$ and $F\mapsto\hF$ establish an isomorphism of posets between $\CQ\text{-}{\bf Dist}(\bbA,\bbB)$ and $\CQ\text{-}{\bf Cat}^{\rm co}(\bbA,\PdB)$, while the maps $\phi\mapsto\laphi$ and $F\mapsto\hF$ establish an isomorphism of posets between $\CQ\text{-}{\bf Dist}(\bbA,\bbB)$ and $\CQ\text{-}{\bf Cat}(\bbB,\PA)$. Together with  Theorem \ref{Isbell_distributor_bijection}, we have the following isomorphisms of posets
\begin{equation} \label{Isbell_isomorphism}
\CQ\text{-}{\bf Dist}(\bbA,\bbB)\cong\CQ\text{-}{\bf Cat}^{\rm co}(\bbA,\PdB)\cong\CQ\text{-}{\bf Cat}(\bbB,\PA)\cong\CQ\text{-}{\bf CCat}^{\rm co}(\PA,\PdB).
\end{equation}

Given a $\CQ$-distributor $\phi:\bbA\olra\bbB$,  it follows from Example \ref{adjunction_closure_interior} that $\dphi\circ\uphi:\PA\lra\PA$ is a $\CQ$-closure operator and $\uphi\circ\dphi:\PdB\lra\PdB$ is a $\CQ$-interior operator. For each $y\in\bbB_0$, since
\begin{equation} \label{phiF_fixed}
\laphi y=\phi(-,y)=\dphi\circ\sYd_{\bbB}y= \dphi\circ\uphi\circ\dphi\circ\sYd_{\bbB}y,
\end{equation}
it follows that $\laphi y=\phi(-,y)$ is closed in the $\CQ$-closure space $(\bbA, \dphi\circ\uphi)$. Dually, for all $x\in\bbA_0$,
\begin{equation} \label{Fphi_fixed}
\raphi x=\phi(x,-)=\uphi\circ\sY_{\bbA}x=\uphi\circ\dphi\circ\uphi\circ\sY_{\bbA}x
\end{equation}
is a fixed point of the $\CQ$-interior operator $\uphi\circ\dphi$. These facts will be used in the proofs of Theorem \ref{F_U_adjunction} and Theorem \ref{complte_category_sup_dense}.

\begin{prop} \label{F_G_info_F_uphi_dphi_continuous}
Let $(F,G):\phi\lra\psi$ be an infomorphism between $\CQ$-distributors $\phi:\bbA\olra\bbB$ and $\psi:\bbA'\olra\bbB'$. Then $F:(\bbA,\uphi\circ\dphi)\lra(\bbA',\upsi\circ\dpsi)$ is a continuous $\CQ$-functor.
\end{prop}

\begin{proof}
Consider the following diagram:
$$\bfig
\square|alrb|[\PA`\PdB`\PA'`\PdB';\uphi`F^{\ra}`G^{\La}`\upsi]
\square(500,0)/>``>`>/[\PdB`\PA`\PdB'`\PA';\dphi``F^{\ra}`\dpsi]
\efig$$
We must prove $F^{\ra}\circ\dphi\circ\uphi\leq\dpsi\circ\upsi\circ F^{\ra}$. To this end, it suffices to check that

(a) the left square commutes if and only if $(F,G):\phi\lra\psi$ is an infomorphism; and

(b) $F^{\ra}\circ\dphi\leq\dpsi\circ G^{\La}$ if and only if $G^{\natural}\circ\phi\leq\psi\circ F_{\natural}$.

For (a), suppose $G^{\La}\circ\uphi=\upsi\circ F^{\ra}$, then for all $x\in\bbA_0$,
\begin{align*}
G^{\natural}\circ\phi(x,-)&=G^{\La}(\phi(x,-))&(\text{by the definition of}\ G^{\La})\\
&=G^{\La}(\uphi\circ\sY_{\bbA}x)&\text{(by Proposition \ref{uphi_dphi_Yoneda})}\\
&=\upsi(F^{\ra}\circ\sY_{\bbA}x)\\
&=\upsi(\sY_{\bbA}x\circ F^{\natural})&(\text{by the definition of}\ F^{\ra})\\
&=\psi\lda(\sY_{\bbA}x\circ F^{\natural})&(\text{by the definition of}\ \upsi)\\
&=(\psi\circ F_{\natural})\lda\bbA(-,x)&(\text{by Proposition \ref{graph_cograph_implication}(2)})\\
&=\psi\circ F_{\natural}(x,-).
\end{align*}

Conversely, if $(F,G):\phi\lra\psi$ is an infomorphism, then for all $\mu\in\PA$,
\begin{align*}
G^{\La}\circ\uphi(\mu)&=G^{\natural}\circ(\phi\lda\mu)&(\text{by the definition of}\ G^{\La}\ \text{and}\ \uphi)\\
&=(G^{\natural}\circ\phi)\lda\mu&(\text{by Proposition \ref{graph_cograph_implication}(3)})\\
&=(\psi\circ F_{\natural})\lda\mu\\
&=\psi\lda(\mu\circ F^{\natural})&(\text{by Proposition \ref{graph_cograph_implication}(2)})\\
&=\upsi\circ F^{\ra}(\mu).
\end{align*}

For (b), suppose $F^{\ra}\circ\dphi\leq\dpsi\circ G^{\La}$, then for all $y'\in\bbB'_0$,
\begin{align*}
G^{\natural}(-,y')\circ\phi&=G_{\natural}(y',-)\rda\phi&(\text{by Proposition \ref{graph_cograph_implication}(1)})\\
&=\dphi(G_{\natural}(y',-))&(\text{by the definition of}\ \dphi)\\
&\leq F^{\la}\circ F^{\ra}\circ\dphi(G_{\natural}(y',-))&(\text{since}\ F^{\ra}\dv F^{\la}:\PA\rhu\PA')\\
&\leq F^{\la}\circ\dpsi\circ G^{\La}(G_{\natural}(y',-))\\
&=F^{\la}\circ\dpsi\circ G^{\La}\circ G^{\Ra}\circ\sYd_{\bbB'}y'&(\text{by the definition of}\ G^{\Ra})\\
&\leq F^{\la}\circ\dpsi\circ\sYd_{\bbB'}y'&\text{(by Inequality (\ref{FLa_FRa_adjuntion}) and (\ref{dphi_contravariant}))}\\
&=F^{\la}(\psi(-,y'))&\text{(by Proposition \ref{uphi_dphi_Yoneda})}\\
&=\psi(-,y')\circ F_{\natural}.&(\text{by the definition of}\ F^{\la})
\end{align*}

Conversely, if $G^{\natural}\circ\phi\leq\psi\circ F_{\natural}$, then for all $\lam\in\PdB$,
\begin{align*}
F^{\ra}\circ\dphi(\lam)&=(\lam\rda\phi)\circ F^{\natural}&(\text{by the definition of}\ F^{\ra}\ \text{and}\ \dphi)\\
&\leq((G^{\natural}\circ\lam)\rda(G^{\natural}\circ\phi))\circ F^{\natural}\\
&\leq((G^{\natural}\circ\lam)\rda(\psi\circ F_{\natural}))\circ F^{\natural}\\
&\leq(G^{\natural}\circ\lam)\rda(\psi\circ F_{\natural}\circ F^{\natural})\\
&\leq(G^{\natural}\circ\lam)\rda\psi&(\text{since}\ F_{\natural}\dv F^{\natural}:\bbA\rhu\bbB\ \text{in}\ \CQ\text{-}{\bf Dist})\\
&=\dpsi\circ G^{\La}(\lam).&(\text{by the definition of}\ \dpsi\ \text{and}\ G^{\La})
\end{align*}
This completes the proof.
\end{proof}

By virtue of Proposition \ref{F_G_info_F_uphi_dphi_continuous} we obtain a functor $\CU:\CQ\text{-}{\bf Info}\lra\CQ\text{-}{\bf Cls}$ that sends an infomorphism
$$(F,G):(\phi:\bbA\olra\bbB)\lra(\psi:\bbA'\olra\bbB')$$
to a continuous $\CQ$-functor
$$F:(\bbA,\dphi\circ\uphi)\lra(\bbA',\dpsi\circ\upsi).$$

Given a $\CQ$-closure space $(\bbA,C)$, define a $\CQ$-distributor $\zeta_C:\bbA\olra C(\PA)$ by
$$\zeta_C(x,\mu)=\mu(x)$$
for all $x\in\bbA_0$ and $\mu\in C(\PA)$. It is clear that $\zeta_C$ is obtained by restricting the domain and the codomain of the $\CQ$-distributor
$$\PdA\olra\PA, \quad (\lam,\mu)\mapsto \mu\circ\lam. $$

Given a continuous $\CQ$-functor $F:(\bbA,C)\lra(\bbB,D)$ between $\CQ$-closure spaces, consider the $\CQ$-functor $F^{\triangleleft}:D(\PB)\lra C(\PA)$ that sends each closed contravariant presheaf $\lam$ to $F^{\triangleleft}(\lam)=F^{\la}(\lam)$. Then similar to Proposition \ref{Y_functor_Cat_Info} one can check that
$$(F,F^{\triangleleft}):(\zeta_C:\bbA\olra C(\PA))\lra(\zeta_D:\bbB\olra D(\PB))$$
is an infomorphism. Thus, we obtain a functor $\CF:\CQ\text{-}{\bf Cls}\lra\CQ\text{-}{\bf Info}$.

\begin{thm} \label{F_U_adjunction}
$\CF:\CQ\text{-}{\bf Cls}\lra\CQ\text{-}{\bf Info}$ is a left adjoint and right inverse of $\CU:\CQ\text{-}{\bf Info}\lra\CQ\text{-}{\bf Cls}$.
\end{thm}

\begin{proof}
{\bf Step 1.} $\CF$ is a right inverse of $\CU$.

For each $\CQ$-closure space $(\bbA,C)$, by the definition of the functor $\CF$, $\CF(\bbA,C)$ is the $\CQ$-distributor $\zeta_C:\bbA\olra C(\PA)$, where $\zeta_C(x,\mu)=\mu(x)$ for all $x\in\bbA_0$ and $\mu\in C(\PA)$. In order to prove $\CU\circ\CF(\bbA,C)=(\bbA,C)$, we show that $C=\zeta_C^\da\circ(\zeta_C)_\ua$.

For all $\mu\in\PA$ and $\lam\in C(\PA)$, since $C$ is a $\CQ$-functor,
$$\lam\lda\mu=\PA(\mu,\lam)\leq\PA(C(\mu),\lam)=\lam\lda C(\mu),$$
and consequently $C(\mu)\leq(\lam\lda\mu)\rda\lam$. Since $C$ is a $\CQ$-closure operator, we have
$$(C(\mu)\lda\mu)\rda C(\mu)\leq 1_{t\mu}\rda C(\mu)=C(\mu),$$
hence
\begin{align*}
C(\mu)&=\bw_{\lam\in C(\PA)}(\lam\lda\mu)\rda\lam\\
&=\bw_{\lam\in C(\PA)}(\zeta_C(-,\lam)\lda\mu)\rda\zeta_C(-,\lam)\\
&=\bw_{\lam\in C(\PA)}(\zeta_C)_\ua(\mu)(\lam)\rda\zeta_C(-,\lam)\\
&=\zeta_C^\da\circ(\zeta_C)_\ua(\mu),
\end{align*}
as required.

{\bf Step 2.} $\CF$ is a left adjoint of $\CU$.

For each $\CQ$-closure space $(\bbA,C)$, ${\rm id}_{(\bbA,C)}:(\bbA,C)\lra\CU\circ\CF(\bbA,C)$ is clearly a continuous $\CQ$-functor and $\{{\rm id}_{(\bbA,C)}\}$ is a natural transformation from the identity functor on $\CQ$-{\bf Cls} to $\CU\circ\CF$. Thus, it remains to show that for each $\CQ$-distributor $\psi:\bbA'\olra\bbB'$ and each continuous $\CQ$-functor $H:(\bbA,C)\lra(\bbA',\dpsi\circ\upsi)$, there is a unique infomorphism
$$(F,G):\CF(\bbA,C)\lra(\psi:\bbA'\olra\bbB')$$
such that the diagram
$$\bfig
\qtriangle<700,500>[(\bbA,C)`\CU\circ\CF(\bbA,C)`(\bbA', \dpsi\circ\upsi);{\rm id}_{(\bbA,C)}`H`\CU(F,G)]
\efig$$
is commutative.

By definition, $\CF(\bbA,C)=\zeta_C:\bbA\olra C(\PA)$ and $\CU(F,G)=F$, where $\zeta_C(x,\mu)=\mu(x)$. Thus, we only need to show that there is a unique $\CQ$-functor $G:\bbB'\lra C(\PA)$ such that
$$(H,G):(\zeta_C:\bbA\olra C(\PA))\lra(\psi:\bbA'\olra\bbB')$$
is an infomorphism.

Let $G=H^{\triangleleft}\circ\lapsi:\bbB'\lra C(\PA)$. That $G$ is well-defined follows from the fact that $\lapsi y'\in\dpsi\circ\upsi(\PA')$ for all $y'\in\bbB'_0$ by Equation (\ref{phiF_fixed}) and  that $H:(\bbA,C)\lra(\bbA',\dpsi\circ\upsi)$ is  continuous.   Now we check that
$$(H,G):(\zeta_C:\bbA\olra C(\PA))\lra(\psi:\bbA'\olra\bbB')$$
is an infomorphism. This is easy since
$$\zeta_C(x,Gy')=(Gy')(x)=H^{\triangleleft}\circ\lapsi(y')(x) =\lapsi(y')(Hx)=\psi(Hx,y')$$
for all $x\in\bbA_0$ and $y'\in\bbB'_0$. This proves the existence of $G$.

To see the uniqueness of $G$, suppose that $G':\bbB'\lra C(\PA)$ is another $\CQ$-functor such that
$$(H,G'):(\zeta_C:\bbA\olra C(\PA))\lra(\psi:\bbA'\olra\bbB')$$
is an infomorphism. Then for all $x\in\bbA_0$ and $y'\in\bbB'_0$,
$$(G'y')(x)=\zeta_C(x,G'y')=\psi(Hx,y')=\lapsi(y')(Hx) =H^{\triangleleft}\circ\lapsi(y')(x)=(Gy')(x),$$
hence $G'=G$.
\end{proof}

\begin{cor} \label{Cls_coreflective_Info}
The category $\CQ$-{\bf Cls} is a coreflective subcategory of $\CQ$-{\bf Info}.
\end{cor}
The composition
$$\CM=\CT\circ\CU:\CQ\text{-}{\bf Info}\lra(\CQ\text{-}{\bf CCat})_{\skel}$$
sends a $\CQ$-distributor $\phi:\bbA\olra\bbB$ to a complete $\CQ$-category $\dphi\circ\uphi(\PA)$. Conversely, since $\CF$ is a right inverse of $\CU$ (Theorem \ref{F_U_adjunction}) and $\CT$ is a left inverse of $\CD$ (up to isomorphism, Theorem \ref{T_D_adjunction}),  we have the following

\begin{thm} \label{complete_category_fca}
Every skeletal complete $\CQ$-category is isomorphic to $\CM(\phi)$ for some $\CQ$-distributor $\phi$.
\end{thm}

The following proposition shows that the free cocompletion functor of $\CQ$-categories factors through the functor $\CM$.

\begin{prop} \label{M_Yoneda_PA}
The diagram
$$\bfig
\qtriangle<700,500>[\CQ\text{-}{\bf Cat}`\CQ\text{-}{\bf Info}`\CQ\text{-}{\bf CCat};\BY`\CP`\CM]
\efig$$
commutes.
\end{prop}

\begin{proof}
First, $\CM((\sY_{\bbA})_{\natural})= ((\sY_{\bbA})_{\natural})^{\da}\circ ((\sY_{\bbA})_{\natural})_{\ua}(\PA)=\PA$ for each $\CQ$-category $\bbA$. To see this, it suffices to check that
$$\mu= ((\sY_{\bbA})_{\natural})^{\da}\circ ((\sY_{\bbA})_{\natural})_{\ua}(\mu)= ((\sY_{\bbA})_{\natural}\lda\mu)\rda(\sY_{\bbA})_{\natural}$$
for all $\mu\in\PA$. On one hand, by Yoneda lemma we have
$$(\sY_{\bbA})_{\natural}\lda\mu=(\sY_{\bbA})_{\natural}\lda(\sY_{\bbA})_{\natural}(-,\mu)\geq\PA(\mu,-),$$
thus
$$((\sY_{\bbA})_{\natural}\lda\mu)\rda(\sY_{\bbA})_{\natural}\leq\PA(\mu,-)\rda(\sY_{\bbA})_{\natural}=(\sY_{\bbA})_{\natural}(-,\mu)=\mu.$$
On the other hand, $\mu\leq((\sY_{\bbA})_{\natural}\lda\mu)\rda(\sY_{\bbA})_{\natural}$ holds trivially.

Second, it is trivial that for each $\CQ$-functor $F:\bbA\lra\bbB$,
$$\CM\circ\BY(F)=F^\ra=\CP(F).$$

Therefore, the conclusion holds.
\end{proof}

Corollary \ref{Cls_coreflective_Info} says that the category $\CQ$-{\bf Cls} is a coreflective subcategory of $\CQ$-{\bf Info}. In the following we show that $\CQ$-{\bf Cls} is equivalent to  a subcategory of $\CQ$-{\bf Info}. This equivalence is a generalization of that between closure spaces and state property systems in \cite{Aerts1999}.

\begin{defn}
A $\CQ$-state property system is a triple $(\bbA,\bbB,\phi)$, where $\bbA$ is a $\CQ$-category, $\bbB$ is a skeletal complete $\CQ$-category and $\phi:\bbA\olra\bbB$ is a $\CQ$-distributor, such that
\begin{itemize}
\item[\rm (1)] $\phi(-,{\inf}_{\bbB}\lam)=\lam\rda\phi$ for all $\lam\in\PdB$,
\item[\rm (2)] $\bbB(y,y')=\phi(-,y')\lda\phi(-,y)$ for all $y,y'\in\bbB_0$.
\end{itemize}
\end{defn}

$\CQ$-state property systems and infomorphisms constitute a category $\CQ$-{\bf Sp}, which is a subcategory of $\CQ$-{\bf Info}.

\begin{exmp}
For each $\CQ$-closure space $(\bbA,C)$, 
$(\bbA,C(\PA),\zeta_C)$ is a $\CQ$-state property system.
First, for all $\Psi\in\CPd(C(\PA))$, it follows from Example \ref{PX_PA_complete} and Equation (\ref{closure_system_infimum})  that
\begin{align*}
\zeta_C(-,{\inf}_{C(\PA)}\Psi)&= {\inf}_{C(\PA)}\Psi\\
&= \bw_{\mu\in C(\PA)}\Psi(\mu)\rda\mu\\
&=\bw_{\mu\in C(\PA)}\Psi(\mu)\rda\zeta_C(-,\mu)\\
&=\Psi\rda\zeta_C.
\end{align*}
Second, it is trivial that
$$C(\PA)(\mu,\lam)=\lam\lda\mu=\zeta_C(-,\lam) \lda\zeta_C(-,\mu)$$
for all $\mu,\lam\in C(\PA)$.
\end{exmp}

Therefore, the codomain of the functor $\CF:\CQ\text{-}{\bf Cls}\lra\CQ\text{-}{\bf Info}$ can be restricted to the subcategory $\CQ$-{\bf Sp}.

\begin{thm} \label{F_U_equivalence}
The functors $\CF:\CQ\text{-}{\bf Cls}\lra\CQ\text{-}{\bf Sp}$ and $\CU:\CQ\text{-}{\bf Sp}\lra\CQ\text{-}{\bf Cls}$ establish an equivalence of categories.
\end{thm}

\begin{proof}
It is shown in Theorem \ref{F_U_adjunction}  that $\CU\circ\CF={\bf id}_{\CQ\text{-}{\bf Cls}}$, so, it suffices to prove that $\CF\circ\CU\cong {\bf id}_{\CQ\text{-}{\bf Sp}}$.

Given a $\CQ$-state property system $(\bbA,\bbB,\phi)$, we have by definition
$$\CF\circ\CU(\bbA,\bbB,\phi)=(\bbA,\dphi\circ\uphi(\PA), \zeta_{\dphi\circ\uphi}).$$
By virtue of Equation (\ref{phiF_fixed}), the images of the $\CQ$-functor $\laphi:\bbB\lra\PA$ are contained in $\dphi\circ\uphi(\PA)$, so, it can be viewed as a $\CQ$-functor $\laphi:\bbB\lra\dphi\circ\uphi(\PA)$. Since for any $x\in\bbA_0$ and $y\in\bbB_0$,
$$\phi(x,y)=(\laphi y)(x)=\zeta_{\dphi\circ\uphi}(x,\laphi y),$$
it follows that
$\eta_{\phi}=(1_{\bbA},\laphi)$ is an infomorphism from $\zeta_{\dphi\circ\uphi}: \bbA\olra\dphi\circ\uphi(\PA)$ to $\phi:\bbA\olra\bbB$.
Hence $\eta_{\phi}$ is a morphism from  $\CF\circ\CU(\bbA,\bbB,\phi)$ to  $(\bbA,\bbB,\phi)$ in $\CQ$-{\bf Sp}.
We claim that $\eta$ is a natural isomorphism from $\CF\circ\CU$ to the identity functor ${\bf id}_{\CQ\text{-}{\bf Sp}}$.

Firstly,  $\eta_{\phi}$ is an isomorphism. It suffices to show that
$$\laphi:\bbB\lra\dphi\circ\uphi(\PA)$$
is an isomorphism between $\CQ$-categories.

Since
$$\bbB(y,y')=\phi(-,y')\lda\phi(-,y)=\PA(\laphi y,\laphi y')$$
for all $y,y'\in\bbB_0$, it follows that $\laphi$ is fully faithful. For each $\mu\in\PA$, let $y=\inf_{\bbB}\uphi(\mu)$, then
$$\laphi y=\phi(-,y)=\phi(-,{\inf}_{\bbB}\uphi(\mu))=\uphi(\mu)\rda\phi=\dphi\circ\uphi(\mu),$$
hence $\laphi$ is surjective. Since $\bbB$ is skeletal, we deduce that $\laphi:\bbB\lra\dphi\circ\uphi(\PA)$ is an isomorphism.

Secondly, $\eta$ is natural. For this, we check the commutativity of the following diagram for any infomorphism $(F,G):(\bbA,\bbB,\phi)\lra(\bbA',\bbB',\psi)$ between $\CQ$-state property systems:
$$\bfig
\Square[\CF\circ\CU(\bbA,\bbB,\phi)`(\bbA,\bbB,\phi)`\CF\circ\CU(\bbA',\bbB',\psi)`(\bbA',\bbB',\psi);
(1_{\bbA},\laphi)`(F,F^{\triangleleft})`(F,G)`(1_{\bbA'},\lapsi)] \efig$$
In fact, the equality $F\circ 1_{\bbA}=1_{\bbA'}\circ F$ is clear; and for all $x\in\bbA_0$ and $y'\in\bbB'_0$,
$$\laphi\circ G(y')(x)=\phi(x,Gy')=\psi(Fx,y')=\lapsi(y')(Fx)=F^{\triangleleft}\circ\lapsi(y')(x),$$
thus the conclusion follows.
\end{proof}

Together with Theorem \ref{T_D_adjunction} we have

\begin{cor}
The composition
$$\CT\circ\CU:\CQ\text{-}{\bf Sp}\lra (\CQ\text{-}{\bf CCat})_{\skel}$$
is a left adjoint of
$$\CF\circ\CD:(\CQ\text{-}{\bf CCat})_{\skel}\lra \CQ\text{-}{\bf Sp}.$$
\end{cor}

We end this section with a characterization of  the complete $\CQ$-category $\CM(\phi)$ for a $\CQ$-distributor $\phi:\bbA\olra\bbB$.

Given a $\CQ$-distributor $\phi:\bbA\olra\bbB$, let  $\CM_{\phi}(\bbA,\bbB)$ denote the set of pairs $(\mu,\lam)\in\PA\times\PdB$  such that $\lam=\uphi(\mu)$ and $\mu=\dphi(\lam)$. $\CM_{\phi}(\bbA,\bbB)$ becomes a $\CQ$-typed set if we  assign $t(\mu,\lam)=t\mu=t\lam$. For $(\mu_1,\lam_1),(\mu_2,\lam_2)\in\CM_{\phi}(\bbA,\bbB)$, let
\begin{equation} \label{B_A_B_phi_order}
\CM_{\phi}(\bbA,\bbB)((\mu_1,\lam_1),(\mu_2,\lam_2))=\PA(\mu_1,\mu_2)=\PdB(\lam_1,\lam_2),
\end{equation}
Then $\CM_{\phi}(\bbA,\bbB)$ becomes a $\CQ$-category.

The projection
$$\pi_1:\CM_{\phi}(\bbA,\bbB)\lra\PA,\quad (\mu,\lam)\mapsto\mu$$
is clearly a fully faithful $\CQ$-functor. Since the image of $\pi_1$ is exactly the set of fixed points of the $\CQ$-closure operator $\dphi\circ\uphi:\PA\lra\PA$, we obtain that $\CM_{\phi}(\bbA,\bbB)$ is isomorphic to the complete $\CQ$-category $\CM(\phi)=\dphi\circ\uphi(\PA)$.

Similarly, the projection
$$\pi_2:\CM_{\phi}(\bbA,\bbB)\lra\PdB,\quad (\mu,\lam)\mapsto\lam$$
is also a fully faithful $\CQ$-functor and the image of $\pi_2$ is exactly the set of fixed points of the $\CQ$-interior operator $\uphi\circ\dphi:\PdB\lra\PdB$. Hence $\CM_{\phi}(\bbA,\bbB)$ is also isomorphic to the complete $\CQ$-category $\uphi\circ\dphi(\PdB)$, which is a $\CQ$-interior system of the skeletal complete $\CQ$-category $\PdB$.

Equation (\ref{B_A_B_phi_order}) shows that
$$\uphi:\dphi\circ\uphi(\PA)\lra\uphi\circ\dphi(\PdB)$$
and
$$\dphi:\uphi\circ\dphi(\PdB)\lra\dphi\circ\uphi(\PA)$$
are inverse to each other. Therefore, $\CM(\phi)(=\dphi\circ\uphi(\PA))$,  $\uphi\circ\dphi(\PdB)$  and $\CM_{\phi}(\bbA,\bbB)$ are isomorphic  to each other.

\begin{defn} \label{sup_dense}
A $\CQ$-functor $F:\bbA\lra\bbB$ is $\sup$-dense (resp. $\inf$-dense)  if for any $y\in\bbB_0$, there is some $\mu\in\PA$ (resp. $\lam\in\PdA$) such that $y=\sup_{\bbB}F^{\ra}(\mu)$ (resp. $y=\inf_{\bbB}F^{\Ra}(\lam)$).
\end{defn}

\begin{exmp}
For each $\CQ$-category $\bbA$, the Yoneda embedding $\sY:\bbA\lra\PA$ is $\sup$-dense in $\PA$. Indeed, we have that $\mu=\sup_{\PA}\circ\sY^\ra(\mu)$ for all $\mu\in\PA$ (see Equation (\ref{mu_sup_ymu}) in the proof of Theorem \ref{T_D_adjunction}). Dually, the co-Yoneda embedding $\sYd:\bbA\lra\PdA$ is $\inf$-dense.
\end{exmp}

The following characterization of $\CM_{\phi}(\bbA,\bbB)$ (hence $\CM(\phi)$)  extends Theorem 4.8 in \cite{Lai2009695} to the general setting.

\begin{thm} \label{complte_category_sup_dense}
Given a $\CQ$-distributor $\phi:\bbA\olra\bbB$, a skeletal complete $\CQ$-category $\bbX$ is isomorphic to $\CM_{\phi}(\bbA,\bbB)$  if and only if there exist a $\sup$-dense $\CQ$-functor $F:\bbA\lra\bbX$ and an $\inf$-dense $\CQ$-functor $G:\bbB\lra\bbX$ such that $\phi=G^{\natural}\circ F_{\natural}=\bbX(F-,G-)$.
\end{thm}

\begin{proof}
{\bf Necessity.} It suffices to prove the case $\bbX=\CM_{\phi}(\bbA,\bbB)$. Define $\CQ$-functors $F:\bbA\lra\bbX$ and $G:\bbB\lra\bbX$ by
$$Fa=(\dphi\circ\raphi a,\raphi a),\quad Gb=(\laphi b,\uphi\circ\laphi b),$$
then $F,G$ are well defined by equations (\ref{phiF_fixed}) and (\ref{Fphi_fixed}). It follows that
\begin{align*}
\bbX(F-,G-)&=\PA(\dphi\circ\raphi-,\laphi-)\\
&=\PA(\dphi\circ\raphi-,\dphi\circ\sYd_{\bbB}-)&(\text{by Equation (\ref{phiF_fixed})})\\
&=\PdB(\uphi\circ\dphi\circ\raphi-,\sYd_{\bbB}-)&(\text{by Proposition \ref{uphi-dphi-adjunction}})\\
&=\PdB(\raphi-,\sYd_{\bbB}-)&(\text{by Equation (\ref{Fphi_fixed})})\\
&=(\raphi-)(-)&(\text{by Yoneda lemma})\\
&=\phi.
\end{align*}
Now we show that $F:\bbA\lra\bbX$ is $\sup$-dense. For all $(\mu,\lam),(\mu',\lam')\in\bbX_0$,
\begin{align*}
\bbX((\mu,\lam),(\mu',\lam'))&=\lam'\rda\lam\\
&=\lam'\rda\uphi(\mu)\\
&=\lam'\rda(\phi\lda\mu)\\
&=(\lam'\rda\phi)\lda\mu\\
&=\PdB(\raphi-,\lam')\lda\mu\\
&=\bbX(F-,(\mu',\lam'))\lda\mu&(\text{by Equation (\ref{B_A_B_phi_order})})\\
&=(\bbX(-,(\mu',\lam'))\circ F_{\natural})\lda\mu\\
&=\bbX(-,(\mu',\lam'))\lda(\mu\circ F^{\natural})&(\text{by Proposition \ref{graph_cograph_implication}(2)})\\
&=\bbX(-,(\mu',\lam'))\lda F^{\ra}(\mu),
\end{align*}
thus $(\mu,\lam)={\sup}_{\bbX}\circ F^{\ra}(\mu)$, as desired.

That $G:\bbB\lra\bbX$ is $\inf$-dense can be proved similarly.

{\bf Sufficiency.} We show that the type-preserving function
$$H:\bbX\lra\CM_{\phi}(\bbA,\bbB),\quad Hx=(F_{\natural}(-,x),G^{\natural}(x,-))$$
is an isomorphism of $\CQ$-categories.

{\bf Step 1.} $\bbX=F_{\natural}\lda F_{\natural}=G^{\natural}\rda G^{\natural}$.

For all $x\in\bbX_0$, since $F:\bbA\lra\bbX$ is $\sup$-dense, there is some $\mu\in\PA$ such that $x=\sup_{\bbX}F^{\ra}(\mu)$, thus
\begin{equation} \label{x_sup_Fmu}
\bbX(x,-)=\bbX\lda F^{\ra}(\mu)=\bbX\lda(\mu\circ F^{\natural})=(\bbX\circ F_{\natural})\lda\mu=F_{\natural}\lda\mu,
\end{equation}
where the third equality follows from Proposition \ref{graph_cograph_implication}(2). Consequently
\begin{align*}
\bbX(x,-)&\leq F_{\natural}\lda F_{\natural}(-,x)\\
&\leq(F_{\natural}\lda F_{\natural}(-,x))\circ\bbX(x,x)\\
&=(F_{\natural}\lda F_{\natural}(-,x))\circ(F_{\natural}(-,x)\lda\mu)&(\text{by Equation (\ref{x_sup_Fmu})})\\
&\leq F_{\natural}\lda\mu\\
&=\bbX(x,-),&(\text{by Equation (\ref{x_sup_Fmu})})
\end{align*}
hence $\bbX(x,-)=F_{\natural}\lda F_{\natural}(-,x)=(F_{\natural}\lda F_{\natural})(x,-)$.

Since $G:\bbB\lra\bbX$ is $\inf$-dense, similar calculations lead to $\bbX=G^{\natural}\rda G^{\natural}$.

{\bf Step 2.} $Hx\in\CM_{\phi}(\bbA,\bbB)$ for all $x\in\bbX_0$, thus $H$ is well defined. Indeed,
\begin{align*}
\uphi(F_{\natural}(-,x))&=\phi\lda F_{\natural}(-,x)\\
&=(G^{\natural}\circ F_{\natural})\lda F_{\natural}(-,x)&(\text{since}\ \phi=G^{\natural}\circ F_{\natural})\\
&=G^{\natural}\circ(F_{\natural}\lda F_{\natural}(-,x))&(\text{by Proposition \ref{graph_cograph_implication}(3)})\\
&=G^{\natural}\circ\bbX(x,-)&(\text{by Step 1})\\
&=G^{\natural}(x,-).
\end{align*}
Similar calculation shows that $\dphi(G^{\natural}(x,-))=F_{\natural}(-,x)$. Hence, $Hx\in\CM_{\phi}(\bbA,\bbB)$.

{\bf Step 3.} $H$ is a fully faithful $\CQ$-functor. Indeed, for all $x,x'\in\bbX_0$, by Step 1,
$$\bbX(x,x')=F_{\natural}(-,x')\lda F_{\natural}(-,x)=\PA(F_{\natural}(-,x),F_{\natural}(-,x'))=\CM_{\phi}(\bbA,\bbB)(Hx,Hx').$$

{\bf Step 4.} $H$ is surjective. For each pair $(\mu,\lam)\in\CM_{\phi}(\bbA,\bbB)$, we must show that there is some $x\in\bbX_0$ such that $F_{\natural}(-,x)=\mu$ and $G^{\natural}(x,-)=\lam$. Indeed, let $x=\sup_{\bbX}F^{\ra}(\mu)$, then
\begin{align*}
G^{\natural}(x,-)&=G^{\natural}\circ\bbX(x,-)\\
&=G^{\natural}\circ(F_{\natural}\lda\mu)&(\text{by Equation (\ref{x_sup_Fmu})})\\
&=(G^{\natural}\circ F_{\natural})\lda\mu&(\text{by Proposition \ref{graph_cograph_implication}(3)})\\
&=\phi\lda\mu&(\text{since}\ \phi=G^{\natural}\circ F_{\natural})\\
&=\uphi(\mu)\\
&=\lam,
\end{align*}
and it follows that $F_{\natural}(-,x)=\dphi(G^{\natural}(x,-))=\dphi(\lam)=\mu$.
\end{proof}

\begin{rem} \label{M_example}
 (1) If the quantaloid $\CQ$ has only one object, i.e., $\CQ$ is a unital quantale (in particular the $2$-element Boolean algebra), then a $\CQ$-distributor $\phi:\bbA\olra\bbB$ between discrete $\CQ$-categories is exactly a $\CQ$-valued relations between two sets.\footnote{A $\CQ$-category $\bbA$ is discrete if $\bbA(x,x)= 1_{tx}$ for all $x\in\bbA_0$ and $\bbA(x,y)=\bot_{tx,ty}$ whenever $x\neq y$.} In this case an element $(\mu,\lam)$ in $\CM_{\phi}(\bbA,\bbB)$ is a formal concept of the formal context $(\bbA,\bbB,\phi)$ in the sense of \cite{Radim2004,Ganter:1997:FCA:550737,Shen2013166} and $\CM_{\phi}(\bbA,\bbB)$ is the (fuzzy) formal concept lattice of $(\bbA,\bbB,\phi)$. So, the construction of $\CM({\phi})$ provides an extension of Formal Concept Analysis \cite{Radim2004,Ganter:1997:FCA:550737}.

(2)  If the quantaloid $\CQ$ degenerates to a unital commutative quantale, then  $\CQ$-categories have been treated as quantitative (fuzzy) ordered sets, e.g. \cite{Radim2004,Wagner94solvingrecursive}. In this case, for each $\CQ$-category $\bbA$, $\CM(\bbA)$ is the enriched MacNeille completion of $\bbA$ given in \cite{Radim2004,Wagner94solvingrecursive}. Thus, the construction of $\CM(\phi)$  also generalizes the MacNeille completion of (quantitative) ordered sets.
\end{rem}

\section{Kan adjunctions} \label{Kan_adjunction}

Given a $\CQ$-distributor $\phi:\bbA\olra\bbB$, composing with $\phi$ yields a $\CQ$-functor
$$\phi^*:\PB\lra\PA$$
defined by
$$\phi^*(\lam)=\lam\circ\phi.$$
for all $\lam\in\PB$. Define another $\CQ$-functor
$$\phi_*:\PA\lra\PB$$
by
$$\phi_*(\mu)=\mu\lda\phi.$$

The following propositions \ref{phistar_adjoint} and \ref{phi_star_Yoneda} can be verified in a way similar to that for propositions \ref{uphi-dphi-adjunction} and \ref{uphi_dphi_Yoneda}.

\begin{prop} \label{phistar_adjoint}
$\phi^*\dv\phi_*:\PB\rhu\PA$ in $\CQ$-{\bf Cat}.
\end{prop}

\begin{prop} \label{phi_star_Yoneda}
Let $\phi:\bbA\olra\bbB$ be a $\CQ$-distributor, then $\ulc\phi^*\circ\sY_{\bbB}\urc=\phi$.
\end{prop}

If $\phi:\bbA\rhu\bbB$ is itself a left adjoint $\CQ$-distributor, then $\phi^*$  is not only a left adjoint $\CQ$-functor, but also a right adjoint $\CQ$-functor as asserted in the following

\begin{prop}
$\phi\dv\psi:\bbA\rhu\bbB$ in $\CQ$-{\bf Dist} if and only if $\psi^*\dv\phi^*:\PA\rhu\PB$ in $\CQ$-{\bf Cat}.
\end{prop}

\begin{proof}
{\bf Necessity.} By Proposition \ref{graph_cograph_implication}(2), for all $\mu\in\PA$ and $\lam\in\PB$,
$$\PB(\psi^*(\mu),\lam)=\lam\lda(\mu\circ\psi)=(\lam\circ\phi)\lda\mu=\PA(\mu,\phi^*(\lam)).$$

{\bf Sufficiency.} We must show that $\bbA\leq\psi\circ\phi$ and $\phi\circ\psi\leq\bbB$. Indeed, for all $x\in\bbA_0$ and $y\in\bbB_0$, by Proposition \ref{phi_star_Yoneda},
$$\psi(-,x)\circ\phi=\phi^*(\psi(-,x))=\phi^*\circ\psi^*\circ\sY_{\bbA}x\geq 1_{\PA}\circ\sY_{\bbA}x=\bbA(-,x),$$
$$\phi(-,y)\circ\psi=\psi^*(\phi(-,y))=\psi^*\circ\phi^*\circ\sY_{\bbB}y\leq 1_{\PB}\circ\sY_{\bbB}y=\bbB(-,y).$$
This completes the proof.
\end{proof}

Therefore, for a left adjoint $\CQ$-distributor $\phi$, $\phi^*$ has both a right adjoint $\phi_*$ and a left adjoint $\psi^*$, where $\psi$ is the right adjoint of $\phi$ in $\CQ$-{\bf Dist}. In particular, given a $\CQ$-functor $F:\bbA\lra\bbB$, since the cograph $F^{\natural}:\bbB\olra\bbA$ of $F$ is a right adjoint of the graph $F_\natural :\bbA\olra\bbB$ of $F$, it follows that both $(F^\natural )_*$ and $(F_\natural )^*$ are right adjoints of $(F^\natural )^*$, hence equal to each other. Since $F^\la:\PB\lra\PA$ is the counterpart of the functor $-\circ F$ for $\CQ$-categories, we arrive at the following conclusion which asserts that the adjunction $\phi^*\dv\phi_*$ generalizes  Kan extensions in category theory.

\begin{thm} \label{why_kan}
For each $\CQ$-functor $F:\bbA\lra\bbB$, it holds that
$$(F^{\natural})^*\dv (F^{\natural})_*= F^{\la}=(F_{\natural})^*\dv(F_{\natural})_*.$$
\end{thm}

\begin{rem}
(1) The left Kan extension $(F^{\natural})^*:\PA\lra\PB$ of a $\CQ$-functor $F:\bbA\lra\bbB$ given in Theorem \ref{why_kan} is exactly the pointwise left Kan extension of $\sY_{\bbB}\circ F:\bbA\lra\PB$ along $\sY_{\bbA}:\bbA\lra\PA$ in Stubbe \cite{Stubbe_2005}. Indeed, it can be verified that if the pointwise left Kan extension $\langle F,G\rangle$ of a $\CQ$-functor $F:\bbA\lra\bbB$ along $G:\bbA\lra\bbC$ exists, then for each $c\in\bbC_0$,
$$\langle F,G\rangle(c)=\bbB\lda(F^{\natural})^*(G_{\natural}(-,c)).$$

(2)
Consider the Boolean algebra ${\bf 2}=\{0,1\}$ as an one-object quantaloid. Then every set can be regarded as a discrete ${\bf 2}$-category. Given sets $X$ and $Y$, a distributor $F:X\olra Y$ is essentially a relation from $X$ to $Y$, or a set-valued map $X\lra {\bf 2}^Y$. If we write $F^{\rm op}$ for the dual relation of $F$, then both $F_*$ and  $(F^{\rm op})^*$ are maps from ${\bf 2}^Y$ to ${\bf 2}^X$. Explicitly, for each $V\subseteq Y$,
$$F_*(V)=\{x\in X\mid F(x)\subseteq V\} \ {\rm and}\ (F^{\rm op})^*(V)=\{x\in X\mid F(x)\cap V\not=\emptyset\}. $$

If both $X$ and $Y$ are topological spaces, then the upper and lower semi-continuity of $F$  (as a set-valued map) \cite{Berge1963} can be phrased as follows:  $F$ is upper (resp. lower) semi-continuous  if $F_*(V)$ (resp. $(F^{\rm op})^*(V)$) is open in $X$ whenever $V$ is open in $Y$. In particular, if $F$ is the graph of some map $f:X\lra Y$, then $(F^{\rm op})^*(V)=F_*(V)=f^{-1}(V)$ for all $V\subseteq Y$, hence $f$ is continuous iff $F$ is lower semi-continuous iff $F$ is upper semi-continuous \cite{Berge1963}.
\end{rem}

The following corollary shows that for a fully faithful $\CQ$-functor $F:\bbA\lra\bbB$, both $(F^{\natural})^*$ and $(F_{\natural})_*$ can be regarded as extensions of $F$ \cite{Lawvere1973}.

\begin{cor}
\label{left_right_Kan_equivalent} If $F:\bbA\lra\bbB$ is a fully faithful $\CQ$-functor, then for all $\mu\in\PA$, it holds that $(F^{\natural})^*(\mu)\circ F_{\natural}=\mu$ and $(F_{\natural})_*(\mu)\circ F_{\natural}=\mu$.
\end{cor}

\begin{proof}
The first equality is a reformulation of Proposition \ref{fully_faithful_graph_cograph}(1). For the second equality,
\begin{align*}
(F_{\natural})_*(\mu)\circ F_{\natural}&=(\mu\lda F_{\natural})\circ F_{\natural}\\
&=\mu\lda(F^{\natural}\circ F_{\natural})&(\text{by Proposition \ref{graph_cograph_implication}(4)})\\
&=\mu\lda\bbA&(\text{by Proposition \ref{fully_faithful_graph_cograph}(1)})\\
&=\mu.
\end{align*}
This completes the proof.
\end{proof}

Adjunctions of the form $\phi^*\dv\phi_*:\PB\rhu\PA$ will be called Kan adjunctions by abuse of language. The following theorem states that all adjunctions between $\PB$ and $\PA$ are of this form.

\begin{thm} \label{Kan_distributor_bijection}
The correspondence $\phi\mapsto\phi^*$ is an isomorphism of posets
$$\CQ\text{-}{\bf Dist}(\bbA,\bbB)\cong\CQ\text{-}{\bf CCat}(\PB,\PA).$$
\end{thm}

\begin{proof}
The proof is similar to Theorem \ref{Isbell_distributor_bijection}. The correspondence $G\mapsto \ulc G\circ\sY_{\bbB}\urc$ is an inverse of the correspondence $\phi\mapsto\phi^*$.
\end{proof}

Theorem \ref{Kan_distributor_bijection} adds one more isomorphism of posets to (\ref{Isbell_isomorphism}):
\begin{align*}
\CQ\text{-}{\bf Dist}(\bbA,\bbB)&\cong\CQ\text{-}{\bf Cat}^{\rm co}(\bbA,\PdB)\cong\CQ\text{-}{\bf Cat}(\bbB,\PA)\\
&\cong\CQ\text{-}{\bf CCat}^{\rm co}(\PA,\PdB)\cong\CQ\text{-}{\bf CCat}(\PB,\PA).
\end{align*}

Since $\phi_*\circ\phi^*:\PB\lra\PB$ is a $\CQ$-closure operator for each $\CQ$-distributor $\phi:\bbA\olra\bbB$, it follows that $(\bbB,\phi_*\circ\phi^*)$ is a $\CQ$-closure space.

\begin{prop} \label{F_G_infomorhpism_F_phistar_continuous}
Let $(F,G):(\phi:\bbA\olra\bbB)\lra(\psi:\bbA'\olra\bbB')$ be an infomorphism. Then $G:(\bbB',\psi_*\circ\psi^*)\lra(\bbB,\phi_*\circ\phi^*)$ is a continuous $\CQ$-functor.
\end{prop}

\begin{proof}
Consider the following diagram:
$$\bfig
\square|alrb|[\PB'`\PA'`\PB`\PA;\psi^*`G^{\ra}`F^{\la}`\phi^*]
\square(500,0)/>``>`>/[\PA'`\PB'`\PA`\PB;\psi_*``G^{\ra}`\phi_*]
\efig $$
One must show that $G^{\ra}\circ\psi_*\circ\psi^*\leq\phi_*\circ\phi^*\circ G^{\ra}$. We leave it to the reader to check that the left square  commutes if and only if $(F,G):\phi\lra\psi$ is an infomorphism and that if $G^{\natural}\circ\phi\leq\psi\circ F_{\natural}$ then $G^{\ra}\circ\psi_*\leq\phi_*\circ F^{\la}$.
\end{proof}

By virtue of Proposition \ref{F_G_infomorhpism_F_phistar_continuous} we obtain a functor $\CV:(\CQ\text{-}{\bf Info})^{\rm op}\lra\CQ\text{-}{\bf Cls}$ that sends an infomorphism
$$(F,G):(\phi:\bbA\olra\bbB)\lra(\psi:\bbA'\olra\bbB')$$
to a continuous $\CQ$-functor
$$G:(\bbB',\psi_*\circ\psi^*)\lra(\bbB,\phi_*\circ\phi^*).$$

The composition of
$$\CV:(\CQ\text{-}{\bf Info})^{\rm op}\lra\CQ\text{-}{\bf Cls}$$
and
$$\CT:\CQ{\text -}{\bf Cls}\lra(\CQ\text{-}{\bf CCat})_{\skel}$$
gives a  functor
$$\CK=\CT\circ\CV:(\CQ\text{-}{\bf Info})^{\rm op}\lra(\CQ\text{-}{\bf CCat})_{\skel}$$
that sends each $\CQ$-distributor $\phi:\bbA\olra\bbB$ to the complete $\CQ$-category $\phi_*\circ\phi^*(\PB)$.

The following conclusion asserts that the free cocompletion functor of $\CQ$-categories  factors through  $\CK$.

\begin{prop} \label{K_Yoneda_PA}
If $F:\bbA\lra\bbB$ is a fully faithful $\CQ$-functor, then $\CK(F^{\natural})=\PA$. In particular, the diagram
$$\bfig
\qtriangle<700,500>[\CQ\text{-}{\bf Cat}`(\CQ\text{-}{\bf Info})^{\rm op}`\CQ\text{-}{\bf CCat};\BY^{\dag}`\CP`\CK]
\efig$$
commutes.
\end{prop}

\begin{proof}
In order to see that $\CK(F^{\natural})=(F^{\natural})_*\circ(F^{\natural})^*(\PA)=\PA$, it suffices to check that $(F^{\natural})_*\circ(F^{\natural})^*(\mu)=\mu$ for all $\mu\in\PA$. Indeed,
\begin{align*}
(F^{\natural})_*\circ(F^{\natural})^*(\mu)&=(F_{\natural})^*\circ(F^{\natural})^*(\mu)&(\text{by Theorem\ \ref{why_kan}})\\
&=(F^{\natural}\circ F_{\natural})^*(\mu)\\
&=\bbA^*(\mu)&(\text{by Proposition \ref{fully_faithful_graph_cograph}(1)})\\
&=\mu.
\end{align*}

Furthermore, it is easy to verify that $\CK\circ\BY^{\dag}(G)=G^\ra=\CP(G)$ for each $\CQ$-functor $G:\bbA\lra\bbB$. Thus, the conclusion follows.
\end{proof}

Theorem \ref{complete_category_fca} shows that every skeletal complete $\CQ$-category is of the form $\CM(\phi)$.  It is natural to ask whether every complete $\CQ$-category can be written of the form $\CK(\phi)$ for some $\CQ$-distributor $\phi$. A little surprisingly, this is not true in general. This fact was pointed out in \cite{Lai2009695} in the case that $\CQ$ is a unital commutative quantale. However, the answer is positive for a special kind of quantaloids.

Let $\FD=\{d_A:A\lra A\mid A\in\CQ_0\}$ be a family of morphisms in a quantaloid $\CQ$. $\FD$ is called a {\it cyclic family} \cite{Rosenthal1996} if $d_A\lda f=f\rda d_B$ for all $f\in\CQ(A,B)$. $\FD$ is called a {\it dualizing family} \cite{Rosenthal1996} if $(d_A\lda f)\rda d_A=f=d_B\lda(f\rda d_B)$ for all $f\in\CQ(A,B)$.

A {\it Girard quantaloid} \cite{Rosenthal1996} is a quantaloid with a cyclic dualizing family $\FD$ of morphisms.

\begin{prop} {\rm\cite{Rosenthal1996}} \label{Girard_quantaloid_properties}
Suppose $\CQ$ has a dualizing family
$$\FD=\{d_A:A\lra A\mid A\in\CQ_0\}.$$
Then for all $\CQ$-arrows $f,f_t:A\lra B$, $g:B\lra C$, $h:A\lra C$:
\begin{itemize}
\item[\rm (1)] $g\circ f=d_C\lda(f\rda(g\rda d_C))=((d_A\lda f)\lda g)\rda d_A$.
\item[\rm (2)] $(h\lda f)\rda d_C=f\circ(h\rda d_C)$, $d_A\lda(g\rda h)=(d_A\lda h)\circ g$.
\item[\rm (3)] $(d_B\lda g)\rda f=g\lda (f\rda d_B)$.
\end{itemize}
\end{prop}

Let $\CQ$  be a Girard quantaloid with a cyclic dualizing family
$$\FD=\{d_A:A\lra A\mid A\in\CQ_0\}.$$
For all $f\in\CQ(A,B)$, let
$$\neg f=d_A\lda f=f\rda d_B:B\lra A.$$
Then $\neg\neg f=f$ since $\FD$ is a dualizing family.
For each $\CQ$-category $\bbA$, set
$$(\neg\bbA)(y,x)=\neg\bbA(x,y)$$
for all $x,y\in\bbA_0$. It is easy to verify that $\neg\bbA:\bbA\olra\bbA$ is a $\CQ$-distributor and
$$\FD'=\{\neg\bbA:\bbA\olra\bbA \mid \bbA\in\CQ\text{-}{\bf Dist}\}$$
is a cyclic dualizing family of $\CQ$-{\bf Dist}. Thus
\begin{prop} {\rm\cite{Rosenthal1996}} \label{DistQ_Girard}
If $\CQ$ is a Girard quantaloid, then $\CQ$-{\bf Dist} is a Girard quantaloid.
\end{prop}

Therefore, by assigning $\neg\phi=\neg\bbA\lda\phi=\phi\rda\neg\bbB$ for each $\CQ$-distributor $\phi:\bbA\olra\bbB$, we obtain a  functor $\neg:\CQ\text{-}{\bf Info}\lra(\CQ\text{-}{\bf Info})^{\rm op}$ that sends an infomorphism
$$(F,G):(\phi:\bbA\olra\bbB)\lra(\psi:\bbA'\olra\bbB')$$
to
$$(G,F):(\neg\psi:\bbB'\olra\bbA')\lra(\neg\phi:\bbB\olra\bbA).$$

It is clear that $\neg\circ\neg=1_{\CQ\text{-}{\bf Info}}$. We leave it to the reader to check that $(\neg\phi)(y,x)=\neg\phi(x,y)$ for any distributor $\phi:\bbA\olra\bbB$ and $x\in\bbA_0,y\in\bbB_0$.

\begin{lem} \label{V_G_id}
Suppose $\CQ$ is a Girard quantaloid. Then for any $\CQ$-distributor $\phi:\bbA\olra\bbB$, it holds that $\phi^*=\neg\circ(\neg\phi)_{\ua}$ and $\phi_*=(\neg\phi)^{\da}\circ\neg$.
\end{lem}

\begin{proof}
For all $\lam\in\PB$ and  $\mu\in\PA$, we have
\begin{align*}
\phi^*(\lam)&=\lam\circ\phi\\
&=\lam\circ(\neg\phi\rda\neg\bbA)&(\text{by Proposition \ref{DistQ_Girard}})\\
&=(\neg\phi\lda\lam)\rda\neg\bbA&(\text{by Proposition \ref{Girard_quantaloid_properties}(2)})\\
&=\neg\circ(\neg\phi)_{\ua}(\lam)
\end{align*}
and
\begin{align*}
\phi_*(\mu)&=\mu\lda\phi\\
&=\mu\lda(\neg\phi\rda\neg\bbA)&(\text{by Proposition \ref{DistQ_Girard}})\\
&=(\neg\bbA\lda\mu)\rda\neg\phi&(\text{by Proposition \ref{Girard_quantaloid_properties}(3)})\\
&=\neg\mu\rda\neg\phi&(\text{by Proposition \ref{DistQ_Girard}})\\
&=(\neg\phi)^{\da}\circ\neg\mu.
\end{align*}
The conclusion thus follows.
\end{proof}

\begin{prop} \label{G_V_adjunction}
Suppose $\CQ$ is a Girard quantaloid. Then  $\CV=\CU\circ\neg$ and it has a left adjoint right inverse given by
$$\CG=\neg\circ\CF:\CQ\text{-}{\bf Cls}\lra(\CQ\text{-}{\bf Info})^{\rm op}.$$
Therefore, every skeletal complete $\CQ$-category is isomorphic to $\CK(\phi)$ for some $\CQ$-distributor $\phi$.
\end{prop}

\begin{proof}
This is an immediate consequence of Theorem \ref{F_U_adjunction} and Lemma \ref{V_G_id}.
\end{proof}

\section{Concluding remarks and questions}

Isbell adjunctions and Kan extensions are fundamental constructions in category theory, both of them can be viewed as adjunctions between categories of (contravariant) functors. This paper investigates the functoriality of these constructions in a special setting: categories enriched over a small quantaloid $\CQ$. To this end, infomorphisms (an extension of adjunctions between categories) are introduced to play the role of morphisms between distributors. It is shown that each distributor between categories enriched over a small quantaloid  gives rise to two adjunctions (which are respectively generalizations of Isbell adjunctions and Kan extensions), hence to two monads; and that these two processes are functorial from the category of distributors and  infomorphisms to the category of complete $\CQ$-categories and left adjoints.

This paper is a first step (in a very special setting) to the functoriality of the constructions of Isbell adjunctions and Kan extensions, many things remain to be discovered. We end this paper with two questions.

The definition of infomorphisms is meaningful for distributors between  small categories. The first question is:
Is it possible to establish similar results for distributors between small categories?

The infomorphisms between distributors introduced here can be composed vertically, but not horizontally. So, the second question is:
Is it possible to find a certain kind of morphisms between distributors that can be composed in both directions and behave in a nice way with respect to the construction of Kan extension and Isbell adjunction?


\begin{thebibliography}{ACVS1999}

\bibitem[ACVS1999]{Aerts1999}
D. Aerts, E. Colebunders, A. Van Der Voorde and B. Van Steirteghem.
\newblock State property systems and closure spaces: A study of categorical
  equivalence.
\newblock {\em International Journal of Theoretical Physics}, 38:359--385,
  1999.

\bibitem[Bar1991]{Barr1991}
M. Barr.
\newblock $*$-autonomous categories and linear logic.
\newblock {\em Mathematical Structures in Computer Science}, 1:159--178, 1991.

\bibitem[BS1997]{Barwise1997}
J. Barwise and J. Seligman.
\newblock {\em Information Flow: The Logic of Distributed Systems}, Volume 44
  of {\em Cambridge Tracts in Theoretical Computer Science}.
\newblock Cambridge University Press, Cambridge, 1997.

\bibitem[Bel2004]{Radim2004}
R. B\v{e}lohl\'{a}vek.
\newblock Concept lattices and order in fuzzy logic.
\newblock {\em Annals of Pure and Applied Logic}, 128(1-3):277--298, 2004.

\bibitem[Ben1967]{Benabou1967}
J. B\'{e}nabou.
\newblock Introduction to bicategories.
\newblock In {\em Reports of the Midwest Category Seminar}, Volume 47 of {\em
  Lecture Notes in Mathematics}, pages 1--77. Springer Berlin Heidelberg, 1967.

\bibitem[Ber1963]{Berge1963}
C. Berge.
\newblock {\em Topological Spaces: Including a Treatment of Multi-valued
  Functions, Vector Spaces, and Convexity}.
\newblock Oliver \& Boyd, Edinburgh and London, 1963.

\bibitem[Bor1994]{Borceux1994}
F. Borceux.
\newblock {\em Handbook of Categorical Algebra: Volume 2, Categories and
  Structures}, Volume 51 of {\em Encyclopedia of Mathematics and its
  Applications}.
\newblock Cambridge University Press, Cambridge, 1994.

\bibitem[DL2007]{Day2007651}
B. J. Day and S. Lack.
\newblock Limits of small functors.
\newblock {\em Journal of Pure and Applied Algebra}, 210(3):651--663, 2007.

\bibitem[GW1999]{Ganter:1997:FCA:550737}
B. Ganter and R. Wille.
\newblock {\em Formal Concept Analysis: Mathematical Foundations}.
\newblock Springer, Berlin, 1999.

\bibitem[Gan2007]{Ganter2007}
B. Ganter.
\newblock Relational {Galois} connections.
\newblock In Sergei O. Kuznetsov and Stefan Schmidt, editors, {\em Formal
  Concept Analysis}, Volume 4390 of {\em Lecture Notes in Computer Science},
  pages 1--17. Springer Berlin Heidelberg, 2007.

\bibitem[Hey2010]{Heymans:2010:SQG:2049377}
H. Heymans.
\newblock {\em Sheaves on quantales as generalized metric spaces}.
\newblock PhD thesis, Universiteit Antwerpen, Belgium, 2010.

\bibitem[HS2011]{heymans2011symmetry}
H. Heymans and I. Stubbe.
\newblock Symmetry and cauchy completion of quantaloid-enriched categories.
\newblock {\em Theory and Applications of Categories}, 25(11):276--294, 2011.

\bibitem[Kel1982]{kelly1982basic}
G. M. Kelly.
\newblock {\em Basic concepts of enriched category theory}, Volume 64 of {\em
  London Mathematical Society Lecture Note Series}.
\newblock Cambridge University Press, Cambridge, 1982.

\bibitem[KS2005]{kelly2005notes}
G. M. Kelly and V. Schmitt.
\newblock Notes on enriched categories with colimits of some class.
\newblock {\em Theory and Applications of Categories}, 14(17):399--423, 2005.


\bibitem[LZ2009]{Lai2009695}
H. Lai and D. Zhang.
\newblock Concept lattices of fuzzy contexts: Formal concept analysis vs. rough
  set theory.
\newblock {\em International Journal of Approximate Reasoning}, 50(5):695--707,
  2009.

\bibitem[Law1973]{Lawvere1973}
F. W. Lawvere.
\newblock Metric spaces, generalized logic and closed categories.
\newblock {\em Rendiconti del Seminario Mat\'{e}matico e Fisico di Milano},
  43:135--166, 1973.

\bibitem[Law1986]{lawvere1986taking}
F. W. Lawvere.
\newblock Taking categories seriously.
\newblock {\em Revista Colombiana de Matem{\'a}ticas,  } XX:147--178,
  1986.

\bibitem[Lei2004]{Leinster2004}
T. Leinster.
\newblock {\em Higher Operads, Higher Categories}, Volume 298 of {\em London
  Mathematical Society Lecture Note Series}.
\newblock Cambridge University Press, Cambridge, 2004.


\bibitem[Mac1998]{Lane1998}
S. Mac Lane.
\newblock {\em Categories for the Working Mathematician, Second Edition}, Volume 5 of {\em
  Graduate Texts in Mathematics}.
\newblock Springer, New York, 1998.

\bibitem[Pra1995]{Pratt1995}
V. Pratt.
\newblock Chu spaces and their interpretation as concurrent objects.
\newblock In Jan Leeuwen, editor, {\em Computer Science Today}, Volume 1000 of
  {\em Lecture Notes in Computer Science}, pages 392--405. Springer Berlin
  Heidelberg, 1995.

\bibitem[Ros1996]{Rosenthal1996}
K. I. Rosenthal.
\newblock {\em The Theory of Quantaloids}, Volume 348 of {\em Pitman Research
  Notes in Mathematics Series}.
\newblock Longman, Harlow, 1996.

\bibitem[SZ2013]{Shen2013166}
L. Shen and D. Zhang.
\newblock The concept lattice functors.
\newblock {\em International Journal of Approximate Reasoning}, 54(1):166--183,
  2013.

\bibitem[Stu2005]{Stubbe_2005}
I. Stubbe.
\newblock Categorical structures enriched in a quantaloid: categories,
  distributors and functors.
\newblock {\em Theory and Applications of Categories}, 14(1):1--45, 2005.

\bibitem[Stu2006]{Stubbe_2006}
I. Stubbe.
\newblock Categorical structures enriched in a quantaloid: tensored and
  cotensored categories.
\newblock {\em Theory and Applications of Categories}, 16(14):283--306, 2006.

\bibitem[Wag1994]{Wagner94solvingrecursive}
K. R. Wagner.
\newblock {\em Solving Recursive Domain Equations with Enriched Categories}.
\newblock PhD thesis, Carnegie Mellon University, Pittsburgh, 1994.

\end{thebibliography}
\end{document}